\documentclass[11pt,a4paper]{article}
\usepackage[T1]{fontenc}
\usepackage{lmodern}
\usepackage[margin=28mm]{geometry}
\usepackage{amsmath,amssymb,amsthm,mathtools}
\usepackage{microtype}
\usepackage{enumitem}
\usepackage[hidelinks]{hyperref}
\usepackage{bookmark}
\numberwithin{equation}{section}
\newtheorem{theorem}{Theorem}[section]
\newtheorem{proposition}[theorem]{Proposition}
\newtheorem{lemma}[theorem]{Lemma}

\theoremstyle{remark}
\newtheorem{remark}[theorem]{Remark}
\theoremstyle{definition}
\newtheorem{condition}[theorem]{Condition}
\newtheorem{definition}[theorem]{Definition}
\newcommand{\R}{\mathbb R}
\newcommand{\E}{\mathbb E}
\newcommand{\PP}{\mathbb P}
\newcommand{\F}{\mathcal F}
\newcommand{\dd}{\mathrm d}
\newcommand{\norm}[1]{\left\|#1\right\|}
\newcommand{\np}[2]{\left\|#1\right\|_{#2}}
\newcommand{\e}{\varepsilon}
\setlist{topsep=4pt,itemsep=3pt}
\allowdisplaybreaks[2]

\hypersetup{pdftitle={Local Linearization and Exact Oscillations for the Nonlinear  
SHE with Spatially Rough Noise},pdfauthor={Bin Qian; Mengxin Tong; Ran Wang; Yimin Xiao}}

\title{Local Linearization and Exact Oscillations for the Nonlinear \\
Stochastic Heat Equation with Spatially Rough Noise}

\author{
Bin Qian$^{1}$, Mengxin Tong$^{2}$, Ran Wang$^{2}$, and Yimin Xiao$^{3}$\\[5pt]
\small $^{1}$Department of Mathematics and Statistics,\\
\small Suzhou University of Technology, Changshu, Jiangsu 215500, China\\[3pt]
\small $^{2}$School of Mathematics and Statistics,
Wuhan University, Wuhan 430072, China\\[3pt]
\small $^{3}$Department of Statistics and Probability,\\
\small Michigan State University, East Lansing, MI 48824, USA\\[5pt]
\small\textit{E-mail:}
\href{mailto:binqiancn@126.com}{\texttt{binqiancn@126.com}},
\href{mailto:mengxintong@whu.edu.cn}{\texttt{mengxintong@whu.edu.cn}},\\
\small
\href{mailto:rwang@whu.edu.cn}{\texttt{rwang@whu.edu.cn}},
\href{mailto:xiaoy@msu.edu}{\texttt{xiaoy@msu.edu}}
}

\date{}

\begin{document}
\maketitle

\begin{abstract}
Consider the one-dimensional nonlinear stochastic heat equation 
$$
\frac{\partial u(t,x)}{\partial t}
=
\frac{\partial^2 u(t,x)}{\partial x^2}
+
\sigma(u(t,x))\dot W(t,x),
\qquad t>0,\ x\in\mathbb R,
$$
where $\dot W$ is a Gaussian noise that is white in time and fractional in space with Hurst parameter $H\in(1/4,1/2)$. Hu et al. ({\it Ann. Probab.} {\bf 45} (2017), 4561--4616) established existence and uniqueness of the solution 
under the assumptions that $\sigma$ is differentiable,
both $\sigma$ and $\sigma'$ are Lipschitz continuous,
and $\sigma(0)=0$. Hu and Wang ({\it Ann. Inst. Henri Poincar\'e Probab. Stat.} {\bf 58} (2022), 379--423) subsequently removed the technical condition $\sigma(0)=0$ by developing the well-posedness theory in a suitable weighted space associated with the    weight $\lambda(x)= c_H (1 + |x|^2)^{H-1}$. 

In this weighted framework, which allows $\sigma(0)\ne0$, we establish temporal and spatial local linearizations:   small increments of the nonlinear solution are approximated by the corresponding increments of the solution to the additive
equation, multiplied by the local coefficient $\sigma(u(t,x))$.
The main difficulty lies in handling the nonlocal fractional-difference structure induced by the rough spatial noise. We address this difficulty using weighted heat-kernel increment
estimates.  These approximations yield exact L\'evy-type uniform moduli of continuity in both time and space, together with Khinchin and Chung laws of the iterated logarithm.
\end{abstract}



\section{Introduction}\label{sec:intro}

Consider  the one-dimensional nonlinear stochastic heat equation (SHE, for short)
\begin{equation}\label{eq:SHE}
 \partial_tu(t,x)=\partial_x^2u(t,x)+\sigma(u(t,x))\dot W(t,x),
 \qquad t>0,\ x\in\R,
\end{equation}
with initial data $u(0,\cdot)=u_0$, where $\dot W$ is white in time
and fractional in space with $H\in(1/4,1/2)$. More precisely, $W$ is
a centered Gaussian field with covariance
$$
 \E[W(t,x)W(s,y)]
 =\tfrac12(s\wedge t)(|x|^{2H}+|y|^{2H}-|x-y|^{2H}),
$$
and $\dot W=\partial_t\partial_xW$ is understood as a distribution.
Its spatial spectral measure is
\begin{equation}\label{eq:mu}
 \mu(\dd\xi)=c_{1,1}|\xi|^{1-2H}\dd\xi,
\end{equation}
where
\begin{equation}\label{eq:c11}
 c_{1,1}=\frac1{2\pi}\Gamma(2H+1)\sin(\pi H).
\end{equation} 
 The spatial covariance $\Lambda(x-y)$ can be formally expressed as
$$\Lambda(x-y)=H(2H-1) |x-y|^{2H-2}.$$
For $H\in(1/4,1/2)$, this formal covariance is neither locally
integrable nor nonnegative, and hence does not define a nonnegative
covariance kernel in the classical sense. Consequently, the standard arguments for spatially homogeneous noise
in~\cite{DaPrato2014,Dalang1999,DalangQuerSardanyons2011} do not apply
directly in this rough-noise setting.

Stochastic partial differential equations  driven by spatially rough noise with \(H\in(1/4,1/2)\) have been extensively studied; see \cite{hu19} for a survey. For affine coefficients \(\sigma(x)=ax+b\), Balan et al.~\cite{BJQ2015} established existence and uniqueness of mild solutions to \eqref{eq:SHE} by Fourier-analytic methods, and subsequently proved H\"older continuity in \cite{BJQ2016}.   Hu et al.~\cite{HHLNT2017} established well-posedness   under the assumptions that $\sigma$ is differentiable,
both $\sigma$ and $\sigma'$ are Lipschitz continuous,
and $\sigma(0)=0$. Liu and Mao~\cite{LM2022} studied well-posedness and intermittency under similar assumptions for the SHE with nonlocal fractional differential operators.  Hu and Wang~\cite{HW2022} subsequently removed the assumption
$\sigma(0)=0$ by establishing well-posedness in a suitable weighted space associated with the   weight
$\lambda(x)=c_H(1+|x|^2)^{H-1}$; see Proposition~\ref{prop:solution} below.
Related well-posedness results for stochastic wave equations driven by spatially fractional noise were obtained in \cite{LHW2022,SongSX2020}.

For space--time white noise, corresponding to $H=\frac{1}{2}$,
Khoshnevisan et al.~\cite{KSXZ2014} established a temporal local
linearization of the form
$$
u(t+h,x)-u(t,x)
\approx
\sigma(u(t,x))\bigl[U(t+h,x)-U(t,x)\bigr],
\qquad h\downarrow0,
$$
where $U$ denotes the solution to the additive SHE
(that is, $\sigma\equiv1$) driven by the same noise.
Related local approximation estimates were obtained in
\cite{HP15} under additional smoothness assumptions on $\sigma$.
For spatial increments, Foondun et al.~\cite{FKM2015} established
the analogous local linearization
$$
u(t,x+h)-u(t,x)
\approx
\sigma(u(t,x))\bigl[U(t,x+h)-U(t,x)\bigr],
\qquad h\downarrow0.
$$
Thus, the leading small-scale fluctuation of the nonlinear solution
is given by the corresponding Gaussian increment multiplied by
the local coefficient $\sigma(u(t,x))$, with a remainder of strictly
higher order.  This approximation allows the exact local oscillation properties
of the additive solution to be transferred to the nonlinear solution.  For applications of such local approximations to the study of sample-path regularity,  see, e.g.,
\cite{CHKK19,DalangNualartPu2025,Das2022,HuLee2026,HK2017,KKM23,
KhoshnevisanLee2025,KhoshnevisanLeePuXiao2025,Wang2024,WangXiao2024}.

When \(H\in(1/4,1/2)\), Qian et al.~\cite{Qian2026} established temporal
local linearization under assumptions including \(\sigma(0)=0\).
Building on this work, we extend the temporal result to the weighted
framework of Hu and Wang~\cite{HW2022}, thereby allowing
\(\sigma(0)\neq0\). More importantly, we establish the corresponding
spatial local linearization in the rough-noise regime. The spatial
argument requires a different approach from that of Foondun
et al.~\cite{FKM2015} for space--time white noise. Specifically, rather than
restricting the stochastic integral to an increment-dependent spatial
neighborhood, we work on the whole spatial line and combine semigroup
cancellation with weighted fractional-difference estimates for the heat
kernel. The main difficulty in both the temporal and spatial analyses
is to control the nonlocal fractional-difference norm
while retaining the weight in the estimates.   The weighted heat-kernel estimates developed in
Appendix  provide the main analytic tools for this
purpose.

 To state the main results, let $U$ denote the solution of the additive equation
\begin{equation}\label{eq:linearSHE}
 \partial_tU=\partial_x^2U+\dot W,\qquad U(0,\cdot)=0.
\end{equation}
The solution $U$ has the mild representation
\begin{equation}\label{eq:linearSHE solution}
 U(t,x)=\int_0^t\int_\R G_{t-s}(x-y)W(\dd s,\dd y),
\end{equation}
where 
$$
 G_t(x)=\frac1{\sqrt{4\pi t}}e^{-x^2/(4t)}.
$$
 
Following Hu and Wang   \cite[Section~4.1]{HW2022}, we introduce the   weight function
\begin{equation}\label{eq weight}
\lambda(x)=c_{H}(1+|x|^2)^{H-1},
\end{equation}
where $c_{H}>0$ is the normalizing constant determined by
$\int_{\R}\lambda(x)\,\dd x=1.$
For $T>0$ and $2\le q<\infty$, let $\mathcal Z^q_{\lambda,T}$
be the space of random fields $f$ on $[0,T]\times\R$ with finite norm
$$
\norm{f}_{\mathcal Z^q_{\lambda,T}}
:=
\sup_{0\le t\le T}
\norm{f(t)}_{L^q_\lambda(\Omega\times\R)}
+
\sup_{0\le t\le T}
\mathcal N^*_{1/2-H,q}f(t),
$$
where
$$
\norm{f(t)}_{L^q_\lambda(\Omega\times\R)}
 :=
\left(
\int_\R \E|f(t,x)|^q\lambda(x)\dd x
\right)^{1/q},
$$
and 
$$
\mathcal N^*_{1/2-H,q}f(t)
:=
\left(
\int_\R
\norm{f(t,\cdot+h)-f(t,\cdot)}_{L^q_\lambda(\Omega\times\R)}^2
|h|^{2H-2}\dd h
\right)^{1/2}. 
$$ 
When $f$ is independent of $t$, the corresponding
space and norm are denoted by
$\mathcal Z^q_{\lambda,0}$ and
$\norm{f}_{\mathcal Z^q_{\lambda,0}}$,
respectively.

\begin{condition}\label{cond:A}
The initial datum $u_0:\Omega\times\R\to\R$ is
$\mathcal F_0\otimes\mathcal B(\R)$-measurable and belongs to
$\mathcal Z^{q_0}_{\lambda,0}$ for some
$q_0>\frac{6}{4H-1}$.
The coefficient $\sigma:\R\to\R$ is globally Lipschitz continuous.
\end{condition}

Set
$$
q_H:=\max\left\{\frac6{4H-1},
\frac{4(1-H)}{1-2H}\right\}.
$$
For $q>q_H$, define
$$
\vartheta_q
:=
\min\left\{
\frac H2,
\frac{1-2H-4(1-H)/q}{4}
\right\},
\qquad
\delta_t(q)
:=
\frac{(2-H)\vartheta_q}{2(1+\vartheta_q)},
$$
and set
$$
\delta_x^*:=\min\left\{H(1-H),\frac12-H\right\}.
$$

Define the directional remainders by
\begin{align}
R^{\mathrm t}(t,x,h)
=&\, u(t+h,x)-u(t,x)
-\sigma(u(t,x))[U(t+h,x)-U(t,x)], \label{eq Rt}\\ 
R^{\mathrm x}(t,x,h)
=&\, u(t,x+h)-u(t,x)
-\sigma(u(t,x))[U(t,x+h)-U(t,x)]. \label{eq Rx}
\end{align}
 
\begin{theorem}\label{thm:local-linearization}
Assume Condition~\ref{cond:A} holds with $q_0>q_H$. Fix
$0<T_0<T<\infty$, $q_H<q\le q_0$, and $1\le p\le q/2$.  
\begin{itemize}
\item[(i)] Fix $0<\delta<\delta_t(q)$. There exist $C<\infty$ and
$\e_0>0$ such that
\begin{equation}\label{t:eq:main-moment}
 \sup_{T_0\le t\le T}\sup_{x\in\R}
 \lambda(x)^{1/q}\np{R^{\mathrm t}(t,x,\e)}p
 \le C|\e|^{H/2+\delta},\qquad 0<|\e|\le \e_0.
\end{equation} \item[(ii)]
 
Fix $0<\delta<\delta_x^*$. There exist $C<\infty$ and
$\e_0>0$ such that
\begin{equation}\label{eq:main}
 \sup_{T_0\le t\le T}\sup_{x\in\R}
 \lambda(x)^{1/q}\np{R^{\mathrm x}(t,x,\e)}p
 \le C|\e|^{H+\delta},\qquad 0<|\e|\le \e_0.
\end{equation}

\end{itemize}
\end{theorem}

Theorem~\ref{thm:local-linearization} shows that the small-scale increments of the nonlinear solution $u$
are approximated, to leading order, by the corresponding increments of the additive solution $U$ multiplied by
the local coefficient $\sigma(u(t,x))$. Together with the pathwise remainder estimates established
in Section~\ref{sec:apps}, this approximation allows exact sample-path properties of $U$ to be transferred to $u$. 
 The exact constants follow from the decompositions of $U$
in terms of fractional Brownian motion (fBm, for short) and  its classical
sample-path laws; see Section~\ref{sec:gaussian-inputs}. For related Gaussian results on exact moduli of continuity,
Chung-type laws of the iterated logarithm (LILs, for short), sharp regularity, and small-ball estimates, see Lee and Xiao~\cite{LeeXiao2023},
Hu and Lee~\cite{HuLeeBounded2026}, and Lee and Zhang~\cite{LeeZhang2026}.

To obtain these almost-sure consequences uniformly on compact sets, we
introduce the following pathwise thresholds:
\begin{align}
\delta_{x,\mathrm{pw}}
&:=\left[
\frac{(Hq_0-1)\delta_x^*-2H-2/q_0}
     {Hq_0+3}
\right]_+,
\label{eq:delta-x-pw}\\
\delta_{t,\mathrm{pw}}
&:=\left[
\frac{(Hq_0-2)\delta_t(q_0)-2H-4/q_0}
     {Hq_0+6}
\right]_+,
\label{eq:delta-t-pw}
\end{align}
where $[a]_+:=\max\{a,0\}$.
These thresholds arise from the grid-interpolation argument in
Lemma~\ref{lemma:pathwise}; in particular, both are positive when
$q_0$ is sufficiently large.

Write $\gamma_H=[\Gamma(2H)/\Gamma(H)]^{1/2}$, and let $b_\alpha$
denote the small-ball constant of a standard fBm $B^\alpha$ with
respect to the supremum norm, defined by
\begin{equation}\label{eq:fbm-small-constant}
 b_\alpha
 :=-\lim_{\varepsilon\downarrow0}\varepsilon^{1/\alpha}
 \log\PP\left\{
 \sup_{0\le z\le1}|B^\alpha(z)|\le\varepsilon
 \right\}
 \in(0,\infty).
\end{equation}
The existence of this constant was established
in~\cite{LiLinde1998}.

\begin{proposition}\label{thm:oscillations}
Assume that Condition~\ref{cond:A} holds with $q_0>q_H$.
For each fixed $(t,x)\in(0,\infty)\times\R$, the following
Khinchin and Chung LILs hold almost surely.
\begin{itemize}
\item[(i)] If $\delta_{x,\mathrm{pw}}>0$, then
\begin{align}
 \limsup_{h\downarrow0}
 \frac{u(t,x+h)-u(t,x)}
      {h^H\sqrt{2\log\log(1/h)}}
 =&\, \frac{|\sigma(u(t,x))|}{\sqrt2},
 \label{eq:spatial-lil}\\
 \liminf_{r\downarrow0}\sup_{0\le h\le r}
 \frac{ |u(t,x+h)-u(t,x)|}
      {(r/\log\log(1/r))^H}
 =&\, \frac{b_H^H}{\sqrt2}|\sigma(u(t,x))|.
 \label{eq:spatial-chung}
\end{align}

\item[(ii)] If $\delta_{t,\mathrm{pw}}>0$, then
\begin{align}
 \limsup_{h\downarrow0}
 \frac{u(t+h,x)-u(t,x)}
      {h^{H/2}\sqrt{2\log\log(1/h)}}
 =&\, \gamma_H|\sigma(u(t,x))|,
 \label{eq:temporal-lil}\\
 \liminf_{r\downarrow0}\sup_{0\le h\le r}
 \frac{ |u(t+h,x)-u(t,x)|}
      {(r/\log\log(1/r))^{H/2}}
 =&\, \gamma_H b_{H/2}^{H/2}|\sigma(u(t,x))|.
 \label{eq:temporal-chung}
\end{align}
\end{itemize}
\end{proposition}

Uniform control of the local-linearization remainders also yields
the exact spatial and temporal uniform moduli of continuity
on compact intervals.

\begin{theorem}\label{thm:uniform-moduli}
Assume that Condition~\ref{cond:A} holds with $q_0>q_H$.
\begin{itemize}
\item[(i)] Assume $\delta_{x,\mathrm{pw}}>0$. For each fixed $t>0$,
with probability one, every
nondegenerate compact interval $J\subset\R$ satisfies
\begin{equation}\label{eq:spatial-modulus}
 \lim_{r\downarrow0}\sup_{\substack{x,y\in J\\ |y-x|\le r}}
 \frac{ 
       |u(t,y)-u(t,x)|}
      {r^H\sqrt{2\log(1/r)}}
 =\frac1{\sqrt2}\max_{y\in J}|\sigma(u(t,y))|.
\end{equation}

\item[(ii)] Assume $\delta_{t,\mathrm{pw}}>0$. For each fixed $x\in\R$,
with probability one, every
nondegenerate compact interval $I\subset(0,\infty)$ satisfies
\begin{equation}\label{eq:temporal-modulus}
 \lim_{r\downarrow0}\sup_{\substack{s,t\in I\\ |t-s|\le r}}
 \frac{ 
       |u(t,x)-u(s,x)|}
      {r^{H/2}\sqrt{2\log(1/r)}}
 =\gamma_H\max_{s\in I}|\sigma(u(s,x))|.
\end{equation}
\end{itemize}
\end{theorem}

 The relationship between our uniform modulus and the coefficient-normalized result of Hu and Lee~\cite[Theorem~1.2]{HuLee2026} is discussed in Remark~\ref{rem:normalized-moduli}.

In Section~\ref{sec:prelim}, we present the stochastic integration
framework and the weighted estimates used throughout the paper.
In Sections~\ref{sec:temporal} and~\ref{sec:spatial}, we establish the
temporal and spatial local linearizations, respectively.
In Section~\ref{sec:apps}, we apply these results to derive exact
uniform moduli of continuity and pointwise laws of the iterated
logarithm.
In Appendix, we collect the auxiliary heat-kernel estimates used in
the proofs.

\section{Weighted framework and preliminary estimates}\label{sec:prelim}

\subsection{Covariance structure and stochastic integration}\label{sec:noise}

We recall some notation from \cite{HHLNT2017}. Let $\mathcal{D}(\mathbb{R})$ denote the space of real-valued infinitely differentiable functions with compact support in $\mathbb{R}$. For $f\in\mathcal{D}(\mathbb{R})$, its Fourier transform is defined by
$$
\widehat f(\xi):=\int_{\mathbb{R}} e^{-i\xi x}f(x)\,\dd x.
$$

Let $\mathcal{H}$ be the Hilbert space obtained by completing $\mathcal{D}(\mathbb{R})$ with respect to the inner product
\begin{equation}\label{eq:inner}
\begin{split}
\langle \phi,\psi\rangle_{\mathcal{H}}
 \coloneqq &\, c_{1,1}\int_{\mathbb{R}}
\widehat{\phi}(\xi)\overline{\widehat{\psi}(\xi)}
|\xi|^{1-2H}\,\dd\xi
 \\
 =&\, H\left(\frac{1}{2}-H\right)
\int_{\mathbb{R}^2}
[\phi(x+y)-\phi(x)]
[\psi(x+y)-\psi(x)]
|y|^{2H-2}\,\dd x\,\dd y,
 \end{split}
\end{equation}
for $\phi,\psi\in\mathcal{D}(\mathbb{R})$. Here,   $H\in(1/4,1/2)$, $c_{1,1}$ is given in \eqref{eq:c11},  and the second identity in \eqref{eq:inner} follows from \cite[Proposition 2.8]{BJQ2015}.
 
 Assume that $(\Omega,\mathcal F,(\mathcal F_t)_{t\ge0},\mathbb P)$ is a complete
filtered probability space satisfying the usual conditions. Let
$\mathcal D(\R_+\times\R)$ denote the space of real-valued infinitely
differentiable functions with compact support on $\R_+\times\R$. The
driving noise is represented by a centered Gaussian family
$\{W(\phi):\phi\in\mathcal D(\R_+\times\R)\}$ with covariance
\begin{equation*}
\mathbb{E}[W(\phi)W(\psi)]
=
c_{1,1}\int_{\R_+\times\R}
\widehat{\phi}(s,\xi)
\overline{\widehat{\psi}(s,\xi)}
|\xi|^{1-2H}\,\dd s\,\dd\xi,
\end{equation*}
where    
$\widehat{\phi}(s,\xi)$ denotes the Fourier transform of $\phi(s,\cdot)$
in the spatial variable.
We assume that $W$ is compatible with the filtration: after extension
by isometry, the family
\[
W_t(\phi):=W(\mathbf 1_{[0,t]}\phi),\qquad \phi\in\mathcal H,
\]
is an $\mathcal H$-cylindrical Wiener process relative to
$(\mathcal F_t)_{t\ge0}$. In particular, $W_t(\phi)$ is adapted  and
\[
\E[W_t(\phi)W_s(\psi)]=(t\wedge s)\langle\phi,\psi\rangle_{\mathcal H}.
\]

Let $\Lambda_H$ denote the space of predictable processes $g$ on $\R_+\times\R$ such  that  
$$
\mathbb{E}\left[\int_{\R_+}\|g(s)\|_{\mathcal H}^2\,\dd s\right]<\infty.
$$
By \cite[Proposition 2.3]{HHLNT2017}, for every $g\in\Lambda_H$, the stochastic integral
$
\int_{\R_+}\int_{\R}g(s,x)\,W(\dd s,\dd x)
$
is well defined as an $L^2(\Omega)$-random variable and satisfies the isometry
\begin{equation}\label{eq:stochastic-isometry}
\mathbb{E}\left[
\left(
\int_{\R_+}\int_{\R}g(s,x)\,W(\dd s,\dd x)
\right)^2
\right]
= \mathbb{E}\left[
\int_{\R_+}\|g(s)\|_{\mathcal H}^2\,\dd s
\right].
\end{equation}

For a random variable $X$ and $p\ge1$, we write
\[
\|X\|_p:=\|X\|_{L^p(\Omega)}.
\] 
For a random  function $f$ defined on $\Omega \times \R$ and $p\ge1$, define the pointwise fractional moment seminorm by
$$
\mathcal N_{1/2-H,p}f(x)
:=
\left(
\int_\R
\|f(x+h)-f(x)\|_{p}^2
|h|^{2H-2}\,\dd h
\right)^{1/2}.
$$ 

We shall use the following BDG-type inequality.
\begin{proposition}[{\cite[Proposition 3.2]{HHLNT2017}}]
\label{prop:bdg}
Let $f\in\Lambda_H$. Then, for every $p\ge2$,
\begin{equation}\label{eq:BDG}
\left\|
\int_0^t\int_{\R}f(s,y)\,W(\dd s,\dd y)
\right\|_{p}
\le
\sqrt{4p}\,c
\left(
\int_0^t\int_{\R}
\bigl[\mathcal N_{1/2-H,p}f(s,y)\bigr]^2
\,\dd y\,\dd s
\right)^{1/2},
\end{equation}
where $c$ is a constant depending only on $H$.
\end{proposition}
 
 For a deterministic kernel $g=\{g_s(y):s\ge0,\ y\in\R\}$
and a predictable field $\xi=\{\xi(s,y):s\ge0,\ y\in\R\}$,
\eqref{eq:BDG} yields, for $p\ge2$ and $0\le t_1<t_2$,
\begin{equation}\label{eq:BDG-product}
\begin{aligned}
&\norm{\int_{t_1}^{t_2}\int_\R
g_s(x-y)\xi(s,y)\,W(\dd s,\dd y)}_{p}^2\\
 \le&\, C_{p,H}\int_{t_1}^{t_2}\int_{\R^2}
|g_s(x-y-h)-g_s(x-y)|^2
\norm{\xi(s,y)}_{p}^2
|h|^{2H-2}\dd y\dd h\dd s\\
&+C_{p,H}\int_{t_1}^{t_2}\int_\R
|g_s(x-y)|^2
\bigl[\mathcal N_{1/2-H,p}\xi(s,y)\bigr]^2\dd y\dd s.
\end{aligned}
\end{equation}

 We use the following notions of mild and weak solutions,
adapted from \cite[Definition~1.4]{HW2022}.

\begin{definition}[{\cite[Definition~1.4]{HW2022}}]
\label{def:weak-solution}
Let $\{u(t,x):t\in[0,T],\ x\in\R\}$ be a real-valued adapted
stochastic process such that, for every $t\in[0,T]$ and $x\in\R$,
$$
G_{t-s}(x-y)\sigma(u(s,y))\mathbf{1}_{[0,t]}(s)
$$
is integrable with respect to $W$.
\begin{itemize}
\item[(i)] We say that $u$ is a
\emph{mild solution} to \eqref{eq:SHE} if, for every
$t\in[0,T]$ and $x\in\R$,
\begin{equation}\label{eq solution nonlinear}
u(t,x)
=
(G_t*u_0)(x)
+
\int_0^t\int_\R
G_{t-s}(x-y)\sigma(u(s,y))\,W(\dd s,\dd y) \ \ \ \  \text{a.s.}
\end{equation}
\item[(ii)]
We say that \eqref{eq:SHE} admits a \emph{weak solution} on $[0,T]$
if there exists a complete filtered probability space
$$
\left(\widetilde{\Omega},\widetilde{\mathcal F},
(\widetilde{\mathcal F}_t)_{t\ge0},\widetilde{\mathbb P}\right)
$$
satisfying the usual conditions and carrying an initial field
$\widetilde u_0$, a Gaussian random field $\widetilde W$,
and an $(\widetilde{\mathcal F}_t)_{t\ge0}$-adapted random field
$$
\{\widetilde u(t,x):t\in[0,T],\ x\in\R\}
$$
 
with the following properties.
The initial field $\widetilde u_0$ is
$\widetilde{\mathcal F}_0\otimes\mathcal B(\R)$-measurable,
and its law as an $L^{q_0}_\lambda(\R)$-valued random variable
coincides with that of $u_0$.
The noise $\widetilde W$ has the same law as $W$,
and its associated $\mathcal H$-cylindrical process is a Wiener process
with respect to $(\widetilde{\mathcal F}_t)_{t\ge0}$.
Finally, $\widetilde u$ is a mild solution to \eqref{eq:SHE}
with initial datum $\widetilde u_0$ and driving noise $\widetilde W$.
\end{itemize}
\end{definition}

 \begin{lemma}\label{lem:initial-smoothing}
\textup{(i)} For the weight function $\lambda$
defined by \eqref{eq weight}, we have
\begin{equation}\label{eq:weight-ratio}
\frac{\lambda(x)}{\lambda(y)}
\le C_H(1+|x-y|)^{2(1-H)},
\qquad x,y\in\R.
\end{equation}
\textup{(ii)} Let $0<T_0<T<\infty$,
$2\le q<\infty$, $1\le p\le q$, and $j=0,1,2$.
If $f\in L^q_\lambda(\Omega\times\R)$, then
\begin{equation}\label{eq:msmooth}
\sup_{T_0\le t\le T}\sup_{x\in\R}
\lambda(x)^{1/q}
\norm{\partial_x^j(G_t*f)(x)}_p
\le
C_{T_0,T,q,H,j}
\norm{f}_{L^q_\lambda(\Omega\times\R)}.
\end{equation}
 
\end{lemma}

\begin{proof}
Part (i) follows from the proof of
\cite[Lemma~2.5]{HW2022}.

For (ii), we first consider deterministic $f$ and put
$q'=q/(q-1)$. 
By H\"older's inequality and (i), we have
$$
\begin{aligned}
&\int_\R
|\partial_x^jG_t(x-y)|\,|f(y)|\dd y\\
\le&\,
\norm{f}_{L^q_\lambda}
\left(
\int_\R
|\partial_x^jG_t(z)|^{q'}
\lambda(x-z)^{-q'/q}\dd z
\right)^{1/q'}\\
\le&\,
C\lambda(x)^{-1/q}\norm{f}_{L^q_\lambda}
\left(
\int_\R
|\partial_x^jG_t(z)|^{q'}
(1+|z|)^{2(1-H)q'/q}\dd z
\right)^{1/q'}\\
\le&\,
C_{T_0,T,q,H,j}
\lambda(x)^{-1/q}\norm{f}_{L^q_\lambda},
\end{aligned}
$$
where the last inequality follows from the uniform Gaussian
decay of the heat-kernel derivatives for $t\in[T_0,T]$.

 For each fixed $x\in\R$, the Gaussian bounds give
$$
\sup_{\substack{t\in[T_0,T]\\ |a|\le1}}
|\partial_x^mG_t(x+a-y)|
\le C e^{-c|x-y|^2},
\qquad 0\le m\le j.
$$
Since $e^{-c|x-y|^2}|f(y)|$ is integrable, we may differentiate
the convolution by differentiating the heat kernel inside the
integral. The preceding estimate then proves \eqref{eq:msmooth}
for deterministic $f$.

For random $f\in L^q_\lambda(\Omega\times\R)$, we have
$f(\omega,\cdot)\in L^q_\lambda(\R)$ almost surely.
Applying the preceding estimate pathwise and using $p\le q$,
we obtain
$$
\begin{aligned}
\norm{\partial_x^j(G_t*f)(x)}_p
&\le
C\lambda(x)^{-1/q}
\left\|\norm{f}_{L^q_\lambda(\R)}\right\|_p\\
&\le
C\lambda(x)^{-1/q}
\left(\E\norm{f}_{L^q_\lambda(\R)}^q\right)^{1/q}\\
&=
C\lambda(x)^{-1/q}
\norm{f}_{L^q_\lambda(\Omega\times\R)}.
\end{aligned}
$$
This proves \eqref{eq:msmooth}.
\end{proof}

 The next proposition follows from  Lemma~\ref{lem:initial-smoothing}(ii) and the estimates and approximation argument in 
\cite[Theorem~1.5, Proposition~4.2, and Lemmas~4.5--4.7]{HW2022}, with the initial field included in the weak-convergence argument. 

\begin{proposition}\label{prop:solution}
Assume Condition~\ref{cond:A} holds. For every $T>0$,
there exists a weak solution $u$ to \eqref{eq:SHE}
on $[0,T]\times\R$ such that
$$
\norm{u}_{\mathcal Z^{q_0}_{\lambda,T}}<\infty.
$$
Moreover, the following estimates hold for any such weak solution.
\begin{itemize}
\item[(i)]
For every $0<T_0<T<\infty$,
$\frac{6}{4H-1}<q\le q_0$, and $1\le p\le q$,
we have
\begin{equation}\label{eq:umom}
\sup_{T_0\le t\le T}\sup_{x\in\R}
\lambda(x)^{1/q}\norm{u(t,x)}_p<\infty.
\end{equation}
\item[(ii)]
If, in addition, $q_H<q\le q_0$, then
for $1\le p\le q$,
\begin{equation}\label{eq:ufrac}
\sup_{T_0\le t\le T}\sup_{x\in\R}
\lambda(x)^{1/q}
\mathcal N_{1/2-H,p}u(t,x)<\infty.
\end{equation}
\end{itemize}
\end{proposition}

 Throughout the paper, $u=\{u(t,x):t\ge0,\ x\in\mathbb R\}$
denotes an arbitrary weak  
 solution to \eqref{eq:SHE} under Condition~\ref{cond:A}
satisfying $\|u\|_{\mathcal Z_{\lambda,T}^{q_0}}<\infty$
for every $T>0$. We work on its underlying filtered probability space and denote
its initial datum and driving noise by $u_0$ and $W$, respectively. 
On the same space, we take $U$ to be the unique mild solution
to the additive equation \eqref{eq:linearSHE}, driven by the same noise $W$.
We use modifications of both $u$ and $U$ that are continuous on $(0,\infty)\times\mathbb R$.

\subsection{H\"older continuity of the solution}
The following Gaussian increment estimate is given in
\cite[Theorem~1.1]{BJQ2016}; see also
\cite[(5.6)]{Qian2026}.

\begin{lemma}\label{lem:U-increment}
For every finite \(p\ge1\),  there exists a constant $C_p>0$ such that for every 
 $s,t\ge0$ and  $x,y\in\R$,
\begin{equation}\label{eq:U-increment-Lp}
 \np{U(t,x)-U(s,y)}p
 \le C_p\bigl(|t-s|^{H/2}+|x-y|^H\bigr).
\end{equation}
\end{lemma}

For $r>0$ and $z,h\in\R$, denote the spatial
heat-kernel increment by
\begin{equation}\label{eq:Dr}
D_r(z,h):=G_r(z+h)-G_r(z).
\end{equation}

Under the present assumptions, the stochastic-convolution
arguments of Hu and Wang~\cite[Section~4.2]{HW2022}
can be adapted to obtain the following pointwise
$L^p(\Omega)$ estimates.
For completeness, we give a detailed proof that makes
the temporal and spatial H\"older exponents explicit.
\begin{proposition}\label{prop:regularity}
Fix $0<T_0<T<\infty$. For every $q_H<q\le q_0$ and $1\le p\le q$, there exists a constant  $C>0$ such that
\begin{equation}\label{eq:exact-holder}
\norm{u(t,x)-u(s,y)}_{p}
\le
C\left(\lambda(x)^{-1/q}+\lambda(y)^{-1/q}\right)
\left(|t-s|^{H/2}+|x-y|^H\right),
\end{equation}
for all $s,t\in[T_0,T]$ and $x,y\in\R$.
\end{proposition}

We first record the following convolution identity,
which follows from stochastic Fubini's theorem.

\begin{lemma}\label{lem:semigroup-fubini}
Fix $r_0>0$ and $q_H<q\le q_0$, and let $k$ be
a Schwartz function on $\R$. Then, for every fixed $x\in\R$,
we have almost surely
$$
\begin{aligned}
&\int_\R k(z)
\bigl[u(r_0,x-z)-(G_{r_0}*u_0)(x-z)\bigr]\dd z\\
 =&\, 
\int_0^{r_0}\int_\R
(k*G_{r_0-s})(x-y)\sigma(u(s,y))W(\dd s,\dd y).
\end{aligned}
$$
\end{lemma}

\begin{proof}
Fix $x\in\R$.
By \eqref{eq:umom} and Lemma~\ref{lem:initial-smoothing}(ii),
$$
\norm{u(r_0,x-z)-(G_{r_0}*u_0)(x-z)}_2\le C\lambda(x-z)^{-1/q}.
$$
Using \eqref{eq:stochastic-isometry}, \eqref{eq solution nonlinear},  
and  
Lemma~\ref{lem:initial-smoothing}(i), we have 
$$
\begin{aligned}
&\int_\R |k(z)|
\left(
\E\int_0^{r_0}
\norm{
G_{r_0-s}(x-z-\cdot)\sigma(u(s,\cdot))
}_{\mathcal H}^2\dd s
\right)^{1/2}\dd z\\
 = &\,
\int_\R |k(z)|\norm{u(r_0,x-z)-(G_{r_0}*u_0)(x-z)}_2\dd z\\
 \le&\,
C\lambda(x)^{-1/q}
\int_\R |k(z)|(1+|z|)^{2(1-H)/q}\dd z
<\infty,
\end{aligned}
$$
where finiteness follows from the rapid decay of $k$.
Thus the convolution on the left-hand side
is well defined in $L^2(\Omega)$, and stochastic Fubini's theorem
yields the stated identity.
\end{proof}
 
 \begin{proof}[Proof of Proposition~\ref{prop:regularity}]
It suffices to prove the estimates for $2\le p\le q$,
since the bounds for lower moments then follow from Jensen's inequality.
Fix $s,t\in[T_0,T]$ and $x,y\in\R$.
By symmetry, we may assume $s\le t$.
Splitting the stochastic convolution at time $T_0/2$, we write
\begin{equation}\label{eq:initial-sharp}
u(t,x)=(G_t*u_0)(x)+I(t,x)+J(t,x),
\end{equation}
where
$$
\begin{aligned}
I(t,x):=&\,
\int_0^{T_0/2}\int_\R
G_{t-r}(x-z)\sigma(u(r,z))W(\dd r,\dd z),\\
J(t,x):=&\,
\int_{T_0/2}^t\int_\R
G_{t-r}(x-z)\sigma(u(r,z))W(\dd r,\dd z).
\end{aligned}
$$
We estimate the increments of these three terms separately.

Since $\partial_tG_t=\partial_x^2G_t$, the fundamental theorem
of calculus, Minkowski's inequality, and
Lemma~\ref{lem:initial-smoothing}(ii) yield
\begin{equation}\label{eq:initial-sharp-expanded}
\begin{aligned}
&\norm{(G_t*u_0)(x)-(G_s*u_0)(y)}_p\\
\le&\,
\int_s^t
\norm{\partial_x^2(G_r*u_0)(x)}_p\dd r
+|x-y|\int_0^1
\norm{\partial_x(G_s*u_0)(y+\theta(x-y))}_p\dd\theta\\
\le&\,
C\lambda(x)^{-1/q}|t-s|
+
C|x-y|\int_0^1
\lambda(y+\theta(x-y))^{-1/q}\dd\theta\\
\le&\,
C\bigl(\lambda(x)^{-1/q}+\lambda(y)^{-1/q}\bigr)
\bigl(|t-s|+|x-y|\bigr).
\end{aligned}
\end{equation}

Set
$$
f(z):=u(T_0/2,z)-(G_{T_0/2}*u_0)(z).
$$
By Lemma~\ref{lem:semigroup-fubini},
$$
I(t,x)-I(s,y)
=
(G_{t-T_0/2}*f)(x)-(G_{s-T_0/2}*f)(y).
$$
Moreover, \eqref{eq:msmooth} and \eqref{eq:umom} imply
$$
\norm{f(z)}_p\le C\lambda(z)^{-1/q},
\qquad z\in\R.
$$
By Minkowski's inequality, the preceding bound on $f$,
\eqref{eq:weight-ratio}, and the Gaussian decay of the
heat-kernel derivatives, we obtain
$$
\begin{aligned}
 \norm{(\partial_aG_a*f)(v)}_p
+\norm{(G_a'*f)(v)}_p 
\le&\,
\int_\R
\bigl(
|\partial_aG_a(v-z)|+|G_a'(v-z)|
\bigr)\norm{f(z)}_p\,\dd z\\
\le&\, C\lambda(v)^{-1/q},
\end{aligned}
$$
uniformly for $a\in[T_0/2,T]$ and $v\in\R$.
Here $G_a'$ denotes the spatial derivative of $G_a$.
Applying the fundamental theorem of calculus to the heat kernels,
as in \eqref{eq:initial-sharp-expanded}, yields
\begin{equation}\label{eq:early-stochastic-sharp}
\norm{I(t,x)-I(s,y)}_p
\le
C\bigl(\lambda(x)^{-1/q}+\lambda(y)^{-1/q}\bigr)
\bigl(|t-s|+|x-y|\bigr).
\end{equation}

For $t\in[T_0,T]$ and $0<|h|\le1$, applying
\eqref{eq:BDG-product} to the spatial increment of $J$ gives
$$
\begin{aligned}
&\norm{J(t,x+h)-J(t,x)}_p^2\\
\le&\,
C\int_{T_0/2}^t\int_{\R^2}
|D_{t-r}(x-z-\eta,h)-D_{t-r}(x-z,h)|^2
\norm{\sigma(u(r,z))}_p^2
|\eta|^{2H-2}\dd z\dd\eta\dd r\\
&+C\int_{T_0/2}^t\int_\R
|D_{t-r}(x-z,h)|^2
\bigl[\mathcal N_{1/2-H,p}\sigma(u(r,\cdot))(z)\bigr]^2
\dd z\dd r.
\end{aligned}
$$
The Lipschitz continuity of $\sigma$, together with
\eqref{eq:umom} and \eqref{eq:ufrac}, yields
\begin{equation}\label{eq:sigma-bounds}
\norm{\sigma(u(r,z))}_p
+
\mathcal N_{1/2-H,p}\sigma(u(r,\cdot))(z)
\le C\lambda(z)^{-1/q},
\end{equation}
uniformly for $r\in[T_0/2,T]$.

For $q>q_H$, write 
\begin{align}\label{eq a}
    a_q:=4(1-H)/q.
\end{align} 
Clearly,   $a_q<1-2H$.  Lemmas~\ref{lem:weightedenergy}
and \ref{lem:singleenergy}(i), together with
\eqref{eq:weight-ratio} and \eqref{eq:sigma-bounds}, yield
\begin{equation}\label{eq:J-space-sharp}
\norm{J(t,x+h)-J(t,x)}_p^2
\le C\lambda(x)^{-2/q}\bigl(|h|^{2H}+|h|\bigr)
\le C\lambda(x)^{-2/q}|h|^{2H},
\end{equation}
where the last inequality uses $0<|h|\le1$ and $2H<1$.

To estimate the temporal increment of $J$, suppose $s<t$
and $t-s\le1$. Write
$$
J(t,x)-J(s,x)=A(x)+B(x),
$$
where
$$
\begin{aligned}
A(x):=&\,
\int_s^t\int_\R
G_{t-r}(x-z)\sigma(u(r,z))W(\dd r,\dd z),\\
B(x):=&\,
\int_{T_0/2}^s\int_\R
\bigl[G_{t-r}(x-z)-G_{s-r}(x-z)\bigr]
\sigma(u(r,z))W(\dd r,\dd z).
\end{aligned}
$$
Using \eqref{eq:BDG-product}, \eqref{eq:sigma-bounds},
\eqref{t:eq:kernel-L2}, and  \eqref{t:eq:kernel-weighted}
with $\vartheta=0$, we obtain
\begin{equation}\label{eq:A-temporal}
\begin{aligned}
\norm{A(x)}_p^2
\le&\,
C\int_s^t\int_{\R^2}
|D_{t-r}(x-z,h)|^2\lambda(z)^{-2/q}
|h|^{2H-2}\dd z\dd h\dd r\\
&+C\int_s^t\int_\R
G_{t-r}(x-z)^2\lambda(z)^{-2/q}\dd z\dd r\\
\le&\,
C\lambda(x)^{-2/q}\int_0^{t-s}
\bigl(r^{H-1}+r^{-1/2}\bigr)\dd r\\
\le&\,
C\lambda(x)^{-2/q}(t-s)^H.
\end{aligned}
\end{equation}

Similarly, \eqref{eq:BDG-product} and \eqref{eq:sigma-bounds} yield
$$
\begin{aligned}
\norm{B(x)}_p^2
\le&\,
C\int_{T_0/2}^s\int_{\R^2}
|D_{t-r}(x-z,h)-D_{s-r}(x-z,h)|^2
\lambda(z)^{-2/q}|h|^{2H-2}\dd z\dd h\dd r\\
&+C\int_{T_0/2}^s\int_\R
|G_{t-r}(x-z)-G_{s-r}(x-z)|^2
\lambda(z)^{-2/q}\dd z\dd r.
\end{aligned}
$$
With $v=s-r$, the triangle inequality and the kernel estimates
\eqref{t:eq:kernel-L2}-\eqref{t:eq:kernel-time-H} give
\begin{equation}\label{eq:B-temporal}
\begin{aligned}
\norm{B(x)}_p^2
\le&\,
C\lambda(x)^{-2/q}\int_0^T
\bigl(v^{H-1}+v^{-1/2}\bigr)
\min\left\{1,\frac{(t-s)^2}{v^2}\right\}\dd v\\
\le&\,
C\lambda(x)^{-2/q}
\bigl((t-s)^H+(t-s)^{1/2}\bigr)\\
\le&\,
C\lambda(x)^{-2/q}(t-s)^H,
\end{aligned}
\end{equation}
where the second inequality follows by splitting the integral
at $v=t-s$.

Combining \eqref{eq:A-temporal} and \eqref{eq:B-temporal} yields
\begin{equation}\label{eq:J-time-sharp}
\norm{J(t,x)-J(s,x)}_p
\le C\lambda(x)^{-1/q}|t-s|^{H/2}.
\end{equation}
Taking square roots in \eqref{eq:J-space-sharp} and using
\eqref{eq:J-time-sharp}, we obtain, for
$|t-s|\le1$ and $|x-y|\le1$,
$$
\begin{aligned}
\norm{J(t,x)-J(s,y)}_p
\le&\,
\norm{J(t,x)-J(s,x)}_p
+\norm{J(s,x)-J(s,y)}_p\\
\le&\,
C\lambda(x)^{-1/q}|t-s|^{H/2}
+C\lambda(y)^{-1/q}|x-y|^H.
\end{aligned}
$$
Together with the decomposition \eqref{eq:initial-sharp}
and the estimates \eqref{eq:initial-sharp-expanded}
and \eqref{eq:early-stochastic-sharp}, this proves
\eqref{eq:exact-holder} in this case.
If either $|t-s|>1$ or $|x-y|>1$, the conclusion follows
directly from \eqref{eq:umom}.
\end{proof}


\section{Temporal local linearization}\label{sec:temporal}
  
\subsection{Decomposition and estimates for temporal increments}
\label{sec:temporal-errors}

Fix $0<T_0<T<\infty$ and $0<\beta<1$.
For $t\in[T_0,T]$, $x\in\R$, and $\e>0$ with $\e^\beta<t$,
the mild formulation gives
\begin{equation}\label{eq:Jdecomp}
\begin{aligned}
&u(t+\e,x)-u(t,x)
-\bigl[(G_{t+\e}*u_0)(x)-(G_t*u_0)(x)\bigr]\\
=&\, J_1+J_{2,\beta}+J_{3,\beta},
\end{aligned}
\end{equation}
where
$$
\begin{aligned}
J_1:=&\,
\int_t^{t+\e}\int_\R
G_{t+\e-s}(x-y)\sigma(u(s,y))W(\dd s,\dd y),\\
J_{2,\beta}:=&\,
\int_0^{t-\e^\beta}\int_\R
\bigl[G_{t+\e-s}(x-y)-G_{t-s}(x-y)\bigr]
\sigma(u(s,y))W(\dd s,\dd y),\\
J_{3,\beta}:=&\,
\int_{t-\e^\beta}^{t}\int_\R
\bigl[G_{t+\e-s}(x-y)-G_{t-s}(x-y)\bigr]
\sigma(u(s,y))W(\dd s,\dd y).
\end{aligned}
$$
Following~\cite{KSXZ2014,Qian2026,WangXiao2024}, we freeze the
coefficient at $(t,x)$ in $J_1$ and $J_{3,\beta}$ and estimate
$J_{2,\beta}$ separately. Define
\begin{equation}\label{eq:J1approx}
\widetilde J_1
:=\sigma(u(t,x))\int_t^{t+\e}\int_\R
G_{t+\e-s}(x-y)W(\dd s,\dd y)
\end{equation}
and
\begin{equation}\label{eq:J3approx}
\widetilde J_{3,\beta}
:=\sigma(u(t,x))\int_{t-\e^\beta}^{t}\int_\R
\bigl[G_{t+\e-s}(x-y)-G_{t-s}(x-y)\bigr]
W(\dd s,\dd y).
\end{equation}

\begin{lemma}\label{t:lem:errors}
Let $q_H<q\le q_0$, $2\le p\le q/2$, and
$0<\vartheta<\vartheta_q$.
There exists a constant $C<\infty$ such that the following
estimates hold uniformly for $t\in[T_0,T]$, $x\in\R$,
and $0<\e<1$ with $\e^\beta<T_0/4$:
\begin{align}
\norm{J_1-\widetilde J_1}_p
\le&\,
C\lambda(x)^{-1/q}\e^{H/2+\vartheta},
\label{t:eq:J1-bound}\\
\norm{J_{2,\beta}}_p
\le&\,
C\lambda(x)^{-1/q}
\bigl(\e^{1+\beta(H/2-1)}+\e^{1/4}\bigr),
\label{t:eq:J2-bound}\\
\norm{J_{3,\beta}-\widetilde J_{3,\beta}}_p
\le&\,
C\lambda(x)^{-1/q}\e^{\beta(H/2+\vartheta)}.
\label{t:eq:refreeze-bound}
\end{align}
\end{lemma}

\begin{proof}
The condition $0<\vartheta<\vartheta_q$ implies
$$
2\vartheta<H,
\qquad
a_q+4\vartheta<1-2H,
\qquad
H+2\vartheta<1/2.
$$
Here, $a_q$ is defined in \eqref{eq a}.
By the Lipschitz continuity of $\sigma$ and \eqref{eq:ufrac},
\begin{equation}\label{eq:N-seminorm-sigma}
\mathcal N_{1/2-H,p}\sigma(u(s,\cdot))(y)
\le C\lambda(y)^{-1/q},
\qquad s\in[T_0/2,T+1].
\end{equation}

\noindent
\emph{Step 1.}
Set
$$
\xi_1(s,y):=\sigma(u(s,y))-\sigma(u(t,x)),
\qquad t<s\le t+\e.
$$
Since $\sigma(u(t,x))$ is $\F_t$-measurable,
$\xi_1$ is predictable on $(t,t+\e]$.
Moreover, the frozen value does not depend on $y$, so
\eqref{eq:N-seminorm-sigma} gives
$$
\mathcal N_{1/2-H,p}\xi_1(s,\cdot)(y)
=
\mathcal N_{1/2-H,p}\sigma(u(s,\cdot))(y)
\le C\lambda(y)^{-1/q}.
$$
By the Lipschitz continuity of $\sigma$, \eqref{eq:umom}, and
  \eqref{eq:exact-holder},  we have
$$
\norm{\xi_1(s,y)}_p^2
\le
C\bigl(\lambda(y)^{-2/q}+\lambda(x)^{-2/q}\bigr)
\bigl((s-t)^{2\vartheta}+|x-y|^{4\vartheta}\bigr).
$$

Applying \eqref{eq:BDG-product}, we obtain
\begin{equation}\label{eq:J1-error-decomp}
\begin{aligned}
\norm{J_1-\widetilde J_1}_p^2
\le&\,
C\int_t^{t+\e}\int_{\R^2}
|D_{t+\e-s}(x-y,h)|^2
\norm{\xi_1(s,y)}_p^2
|h|^{2H-2}\dd y\dd h\dd s\\
&+C\int_t^{t+\e}\int_\R
G_{t+\e-s}(x-y)^2
\bigl[\mathcal N_{1/2-H,p}\sigma(u(s,\cdot))(y)\bigr]^2
\dd y\dd s.
\end{aligned}
\end{equation}
Substituting the preceding coefficient bounds into
\eqref{eq:J1-error-decomp} and applying
\eqref{t:eq:kernel-L2}-\eqref{t:eq:kernel-weighted}
with $r=t+\e-s$, we obtain
$$
\begin{aligned}
\norm{J_1-\widetilde J_1}_p^2
\le&\,
C\lambda(x)^{-2/q}\int_0^\e
\bigl(
(\e-r)^{2\vartheta}r^{H-1}
+r^{H+2\vartheta-1}
+r^{-1/2}
\bigr)\dd r\\
\le&\,
C\lambda(x)^{-2/q}
\bigl(\e^{H+2\vartheta}+\e^{1/2}\bigr)\\
\le&\,
C\lambda(x)^{-2/q}\e^{H+2\vartheta},
\end{aligned}
$$
where the last inequality uses $\e<1$ and
$H+2\vartheta<1/2$.
Taking square roots proves \eqref{t:eq:J1-bound}.

 \noindent\emph{Step 2.}
Split $J_{2,\beta}=I_1+I_2$, where $I_1$ and $I_2$ are the
integrals over $[0,T_0/2]$ and $[T_0/2,t-\e^\beta]$, respectively.

For $I_1$, the semigroup property and
Lemma~\ref{lem:semigroup-fubini} give
\begin{equation*}
\begin{aligned}
I_1
= \,
\int_\R
\bigl[G_{t+\e-T_0/2}(z)-G_{t-T_0/2}(z)\bigr] 
\bigl[u(T_0/2,x-z)-(G_{T_0/2}*u_0)(x-z)\bigr]\dd z.
\end{aligned}
\end{equation*}
The Gaussian decay of the heat-kernel derivatives implies
\begin{equation}\label{eq:heat-kernel-time-weights}
\sup_{\rho\in[T_0/2,T+1]}
\int_\R
|\partial_\rho G_\rho(z)|
(1+|z|)^{2(1-H)/q}\dd z
<\infty.
\end{equation}
By Minkowski's inequality, \eqref{eq:umom}, and
Lemma~\ref{lem:initial-smoothing}(i)--(ii),
$$
\begin{aligned}
\norm{I_1}_p
\le&\,
\int_\R
\bigl|G_{t+\e-T_0/2}(z)-G_{t-T_0/2}(z)\bigr| 
\bigl\|u(T_0/2,x-z)
-(G_{T_0/2}*u_0)(x-z)\bigr\|_p\dd z\\
\le&\,
C\lambda(x)^{-1/q}
\int_\R
\bigl|G_{t+\e-T_0/2}(z)-G_{t-T_0/2}(z)\bigr|
(1+|z|)^{2(1-H)/q}\dd z\\
\le&\,
C\lambda(x)^{-1/q}\e.
\end{aligned}
$$
Here the last inequality follows from the fundamental theorem
of calculus and \eqref{eq:heat-kernel-time-weights}. 

For $I_2$, applying \eqref{eq:BDG-product} and using
the Lipschitz continuity of $\sigma$, \eqref{eq:umom},
and \eqref{eq:N-seminorm-sigma}, we obtain
\begin{equation}\label{eq:I2-bound}
\begin{aligned}
\norm{I_2}_p^2
\le&\,
C\int_{T_0/2}^{t-\e^\beta}\int_{\R^2}
|D_{t+\e-s}(x-y,h)-D_{t-s}(x-y,h)|^2
\lambda(y)^{-2/q}|h|^{2H-2}\dd y\dd h\dd s\\
&+C\int_{T_0/2}^{t-\e^\beta}\int_\R
|G_{t+\e-s}(x-y)-G_{t-s}(x-y)|^2
\lambda(y)^{-2/q}\dd y\dd s\\
\le&\,
C\lambda(x)^{-2/q}
\bigl(\e^{2+\beta(H-2)}+\e^{1/2}\bigr),
\end{aligned}
\end{equation}
where the last inequality follows from
Lemma~\ref{t:lem:integrated}(i)-(ii), applied with
$r=t-s\ge\e^\beta$.
Taking square roots in \eqref{eq:I2-bound} and combining
the estimates for $I_1$ and $I_2$ by the triangle inequality
proves \eqref{t:eq:J2-bound}. 

\noindent
\emph{Step 3. }
Set $\tau:=t-\e^\beta$. 
Define
$$
Z_{t,\e}
:=\int_\tau^t\int_\R
\bigl[G_{t+\e-s}(x-y)-G_{t-s}(x-y)\bigr]
W(\dd s,\dd y)
$$
and write
$$
J_{3,\beta}-\widetilde J_{3,\beta}=K_1+K_2,
$$
where
$$
\begin{aligned}
K_1:=&\,
J_{3,\beta}-\sigma(u(\tau,x))Z_{t,\e},\\
K_2:=&\,
\bigl[\sigma(u(\tau,x))-\sigma(u(t,x))\bigr]Z_{t,\e}.
\end{aligned}
$$

For $K_1$, set
$$
\xi_3(s,y):=\sigma(u(s,y))-\sigma(u(\tau,x)),
\qquad \tau<s\le t.
$$
As in Step 1, $\xi_3$ is predictable on $(\tau,t]$ and has
the same spatial fractional seminorm as $\sigma(u(s,\cdot))$.
The coefficient increment estimate from that step, together
with $s-\tau\le\e^\beta$, also gives
\begin{equation}\label{Eq:xi3}
\norm{\xi_3(s,y)}_p^2
\le
C\bigl(\lambda(y)^{-2/q}+\lambda(x)^{-2/q}\bigr)
\bigl(\e^{2\beta\vartheta}+|x-y|^{4\vartheta}\bigr).
\end{equation}
Applying \eqref{eq:BDG-product} to $K_1$, we obtain
\begin{equation}\label{eq:I1-error-decomp}
\begin{aligned}
\norm{K_1}_p^2
\le&\,
C\int_\tau^t\int_{\R^2}
|D_{t+\e-s}(x-y,h)-D_{t-s}(x-y,h)|^2
\norm{\xi_3(s,y)}_p^2
|h|^{2H-2}\dd y\dd h\dd s\\
&+C\int_\tau^t\int_\R
|G_{t+\e-s}(x-y)-G_{t-s}(x-y)|^2
\bigl[\mathcal N_{1/2-H,p}\sigma(u(s,\cdot))(y)\bigr]^2
\dd y\dd s.
\end{aligned}
\end{equation}
For the kernel difference in the first integral, we use
$$
\begin{aligned}
 |D_{t+\e-s}(x-y,h)-D_{t-s}(x-y,h)|^2 
\le \,
2|D_{t-s}(x-y,h)|^2
+2|D_{t+\e-s}(x-y,h)|^2.
\end{aligned}
$$
Using this inequality,  
  \eqref{eq:N-seminorm-sigma}, and \eqref{Eq:xi3}, we apply
Lemma~\ref{t:lem:integrated}(ii)-(iii) to the two integrals in
\eqref{eq:I1-error-decomp}, with $r=t-s\in[0,\e^\beta]$, to obtain
$$
 \norm{K_1}_p^2
 \le C\lambda(x)^{-2/q}
 \bigl(\e^{\beta(H+2\vartheta)}+\e^{1/2}\bigr).
$$
Since $\e<1$ and $\beta(H+2\vartheta)<1/2$,
taking square roots yields
$$
\norm{K_1}_p
\le C\lambda(x)^{-1/q}
\e^{\beta(H/2+\vartheta)}.
$$

To estimate $K_2$, we use the Gaussianity of $Z_{t,\e}$ and Lemma~\ref{t:lem:integrated}(iii)
to obtain
$$
 \norm{Z_{t,\e}}_{2p}
 \le C\norm{Z_{t,\e}}_2
 \le C\e^{\beta H/2}.
$$
 By H\"older's inequality and the Lipschitz continuity of $\sigma$, we have 
$$
\begin{aligned}
\norm{K_2}_p
\le&\,
\norm{\sigma(u(\tau,x))-\sigma(u(t,x))}_{2p}
\norm{Z_{t,\e}}_{2p}\\
\le&\,
C\lambda(x)^{-1/q}\e^{\beta H}.
\end{aligned}
$$ 

Combining the estimates for $K_1$ and $K_2$ by the triangle
inequality proves \eqref{t:eq:refreeze-bound}. 
The proof is complete.
\end{proof}

\subsection{Proof of Theorem~\ref{thm:local-linearization}(i)}
\label{sec:temporal-proof}

\begin{proof}
We follow the localization strategy used in
\cite[Proposition~4.6]{KSXZ2014},
\cite[Theorem~1.2]{Qian2026}, and
\cite[Theorem~1.2]{WangXiao2024}.

Fix $0<T_0<T<\infty$, $q_H<q\le q_0$, and $0<\delta<\delta_t(q)$. 
Define
$$
\delta(\vartheta):=\frac{(2-H)\vartheta}{2(1+\vartheta)}.
$$
Since $\delta(\vartheta)$ is continuous and strictly increasing,
with $\delta(\vartheta_q)=\delta_t(q)$, we may choose
$0<\vartheta<\vartheta_q$ such that $\delta<\delta(\vartheta)$.
To balance the error exponents, set
$\beta:=1/(1+\vartheta)\in(0,1)$. Then
\begin{equation}\label{t:eq:balance}
2+\beta(H-2)
=\beta(H+2\vartheta)
=H+2\delta(\vartheta).
\end{equation}
Choose $\e_0\in(0,1)$ sufficiently small that
$\e_0^\beta<T_0/4$.

Let $2\le p\le q/2$ and $0<\e\le\e_0$.
For  $t\in[T_0,T]$ and $x\in\R$, by \eqref{eq:Jdecomp}-\eqref{eq:J3approx}, we have 
\begin{equation}\label{Eq:Rt}
\begin{aligned}
R^{\mathrm t}(t,x,\e)
=&\,
(G_{t+\e}*u_0)(x)-(G_t*u_0)(x) +(J_1-\widetilde J_1)+J_{2,\beta}
 +(J_{3,\beta}-\widetilde J_{3,\beta})\\
&-\sigma(u(t,x))
\int_0^{t-\e^\beta}\int_\R
\bigl[G_{t+\e-s}(x-y)-G_{t-s}(x-y)\bigr]
W(\dd s,\dd y).
\end{aligned}
\end{equation}

For the initial term in \eqref{Eq:Rt}, the fundamental theorem
of calculus, Minkowski's inequality, and
Lemma~\ref{lem:initial-smoothing}(ii) yield
$$
\begin{aligned}
&\sup_{T_0\le t\le T}\sup_{x\in\R}
\lambda(x)^{1/q}
\norm{(G_{t+\e}*u_0)(x)-(G_t*u_0)(x)}_p\\
\le&\,
\e\sup_{T_0\le s\le T+1}\sup_{x\in\R}
\lambda(x)^{1/q}
\norm{\partial_x^2(G_s*u_0)(x)}_p\\
\le&\,C\e,
\end{aligned}
$$
where we used
$\partial_s(G_s*u_0)=\partial_x^2(G_s*u_0)$.

For the last term in (\ref{Eq:Rt}), H\"older's inequality,
\eqref{eq:umom}, Gaussian moment comparison, and the unweighted
version of Lemma~\ref{t:lem:integrated}(i) give
$$
\begin{aligned}
&\left\|
\sigma(u(t,x))
\int_0^{t-\e^\beta}\int_\R
\bigl[G_{t+\e-s}(x-y)-G_{t-s}(x-y)\bigr]
W(\dd s,\dd y)
\right\|_p\\
\le&\,
\norm{\sigma(u(t,x))}_{2p}
\left\|
\int_0^{t-\e^\beta}\int_\R
\bigl[G_{t+\e-s}(x-y)-G_{t-s}(x-y)\bigr]
W(\dd s,\dd y)
\right\|_{2p}\\
\le&\,
C\lambda(x)^{-1/q}\e^{1+\beta(H/2-1)}.
\end{aligned}
$$ 

Combining these bounds with Lemma~\ref{t:lem:errors}
and using \eqref{t:eq:balance}, we obtain
$$
\begin{aligned}
 \sup_{T_0\le t\le T}\sup_{x\in\R}
\lambda(x)^{1/q}\norm{R^{\mathrm t}(t,x,\e)}_p 
\le&\,
C\bigl(
\e
+\e^{H/2+\vartheta}
+\e^{H/2+\delta(\vartheta)}
+\e^{1/4}
\bigr)\\
\le&\,
C\e^{H/2+\delta},
\end{aligned}
$$
where the last inequality uses $0<\e<1$ and
$\delta<\delta(\vartheta)<\vartheta<1/4-H/2$.

This proves \eqref{t:eq:main-moment} for
$2\le p\le q/2$ and $\e > 0$.
After decreasing $\e_0$ if necessary, 
the remainder $R^{\mathrm t}(t,x,\e)$ can be 
 handled similarly for $\e < 0$, 
 using $t+\e$ as the base point.
For $1\le p<2$, the same estimate follows from the 
  monotonicity of $L^p$ norms. The proof is complete. 
\end{proof}

\section{Spatial local linearization}\label{sec:spatial}

\subsection{Decomposition and estimates for spatial increments}
\label{sec:spatialdecomp}
Fix $0<T_0<T<\infty$. 
Let $t\in[T_0,T]$,  $x\in\R$,  and
$0<\rho<1\wedge(T_0/2)$, so that $t-\rho\ge T_0/2$.
For $\e>0$,  we split the stochastic convolutions
at time $t-\rho$ by setting
$$
\begin{aligned}
J_{\mathrm f}
:=&\,
\int_0^{t-\rho}\int_\R
D_{t-s}(x-y,\e)\sigma(u(s,y))\,W(\dd s,\dd y),\\
J_{\mathrm n}
:=&\,
\int_{t-\rho}^t\int_\R
D_{t-s}(x-y,\e)\sigma(u(s,y))\,W(\dd s,\dd y),\\
L_{\mathrm f}
:=&\,
\int_0^{t-\rho}\int_\R
D_{t-s}(x-y,\e)\,W(\dd s,\dd y),\\
L_{\mathrm n}
:=&\,
\int_{t-\rho}^t\int_\R
D_{t-s}(x-y,\e)\,W(\dd s,\dd y).
\end{aligned}
$$
 The mild formulations give
\begin{equation}\label{eq:spatial-u-decomp}
\begin{aligned}
u(t,x+\e)-u(t,x)
-\bigl[(G_t*u_0)(x+\e)-(G_t*u_0)(x)\bigr]
= J_{\mathrm f}+J_{\mathrm n},
\end{aligned}
\end{equation}
and
\begin{equation}\label{eq:spatial-U-decomp}
U(t,x+\e)-U(t,x)
=L_{\mathrm f}+L_{\mathrm n}.
\end{equation}

 To preserve predictability on $(t-\rho,t]$,
we freeze the coefficient in $J_{\mathrm n}$ at $(t-\rho,x)$
and define
\begin{equation}\label{eq:Jn-approx}
\widetilde J_{\mathrm n}
:=\sigma(u(t-\rho,x))L_{\mathrm n}.
\end{equation}

 \begin{lemma}\label{lem:freeze}
Let  $q_H<q\le q_0$ and $2\le p\le q/2$.
There exists a constant $C<\infty$ such that
the following estimates hold uniformly for $t\in[T_0,T]$,
$x\in\R$, $0<\rho<1\wedge(T_0/2)$, and $0<\e\le1$:
\begin{itemize}
\item[(i)]
\begin{equation}\label{eq:freeze}
\lambda(x)^{1/q}
\np{J_{\mathrm n}-\widetilde J_{\mathrm n}}p
\le
C\bigl(
\e^H\rho^{H/2}
+\e^{2H}
+\e^{1/2}
\bigr).
\end{equation}
\item[(ii)]
\begin{equation}\label{eq:freeze-coeff}
\lambda(x)^{1/q}
\np{[\sigma(u(t-\rho,x))-\sigma(u(t,x))]L_{\mathrm n}}p
\le C\e^H\rho^{H/2}.
\end{equation}
\end{itemize}
\end{lemma}

\begin{proof}
 \noindent
\emph{(i)}
For $t-\rho\le s\le t$, set
$$
A_s(y):=\sigma(u(s,y))-\sigma(u(t-\rho,x)).
$$
Since $\sigma(u(t-\rho,x))$ is
$\mathcal F_{t-\rho}$-measurable, $A$ is predictable on
$(t-\rho,t]$, and
$$
J_{\mathrm n}-\widetilde J_{\mathrm n}
=
\int_{t-\rho}^t\int_\R
D_{t-s}(x-y,\e)A_s(y)\,W(\dd s,\dd y).
$$
Since $\rho<1$, the Lipschitz continuity
of $\sigma$ and \eqref{eq:exact-holder}, together with
\eqref{eq:umom} for $|y-x|>1$, give
\begin{equation}\label{eq:As}
\np{A_s(y)}p^2
\le
C\bigl[\lambda(y)^{-2/q}+\lambda(x)^{-2/q}\bigr]
\bigl[\rho^H+(|y-x|^{2H}\wedge1)\bigr].
\end{equation}
Moreover, the frozen value does not depend on $y$, so
the Lipschitz continuity of $\sigma$ and \eqref{eq:ufrac} imply
\begin{align}\label{eq NA}
\mathcal N_{1/2-H,p}A_s(y)
=
\mathcal N_{1/2-H,p}\sigma(u(s,\cdot))(y)
\le C\lambda(y)^{-1/q}.
\end{align}

Applying \eqref{eq:BDG-product}, we obtain
$$
\norm{J_{\mathrm n}-\widetilde J_{\mathrm n}}_p^2
\le C(B_1+B_2),
$$
where
$$
\begin{aligned}
B_1:=&\,
\int_{t-\rho}^t\int_{\R^2}
|D_{t-s}(x-y-h,\e)-D_{t-s}(x-y,\e)|^2
\np{A_s(y+h)}p^2
|h|^{2H-2}\dd h\dd y\dd s,\\
B_2:=&\,
\int_{t-\rho}^t\int_\R
|D_{t-s}(x-y,\e)|^2
\bigl[
\mathcal N_{1/2-H,p}\sigma(u(s,\cdot))(y)
\bigr]^2
\dd y\dd s.
\end{aligned}
$$

For $B_1$, set $r=t-s$ and $z=x-y-h$, so that
$A_s(y+h)=A_{t-r}(x-z)$.
By \eqref{eq:As}, Lemma~\ref{lem:initial-smoothing}(i),
and Lemma~\ref{lem:weightedenergy},
$$
\begin{aligned}
B_1
\le&\,
C\lambda(x)^{-2/q}\int_0^\rho\int_{\R^2}
|D_r(z+h,\e)-D_r(z,\e)|^2(1+|z|)^{a_q} 
\bigl[\rho^H+(|z|^{2H}\wedge1)\bigr]
|h|^{2H-2}\dd h\dd z\dd r\\
\le&\,
C\lambda(x)^{-2/q}
\bigl(\rho^H\e^{2H}+\e^{4H}+\e\bigr).
\end{aligned}
$$
For $B_2$, use $r=t-s$ and $z=x-y$. By \eqref{eq NA},  
Lemmas~\ref{lem:initial-smoothing}(i) and
\ref{lem:singleenergy}(i), we have 
$$
\begin{aligned}
B_2
\le&\,
C\lambda(x)^{-2/q}
\int_0^\rho\int_\R
|D_r(z,\e)|^2(1+|z|)^{a_q}\dd z\dd r\\
\le&\,
C\lambda(x)^{-2/q}\e.
\end{aligned}
$$
Combining the bounds for $B_1$ and $B_2$ and taking
square roots proves \eqref{eq:freeze}.

\noindent
\emph{(ii)}
The Gaussianity of $L_{\mathrm n}$,  \eqref{eq:stochastic-isometry},
and Lemma~\ref{lem:energy} give 
$$
\np{L_{\mathrm n}}{2p}
\le C\np{L_{\mathrm n}}2
\le C\e^H.
$$
 By H\"older's inequality, the Lipschitz continuity of $\sigma$, and \eqref{eq:exact-holder}, we have 
$$
\begin{aligned}
&\lambda(x)^{1/q}
\np{[\sigma(u(t-\rho,x))-\sigma(u(t,x))]L_{\mathrm n}}p\\
\le&\,
\lambda(x)^{1/q}
\np{\sigma(u(t-\rho,x))-\sigma(u(t,x))}{2p} \cdot
\np{L_{\mathrm n}}{2p}\\
\le&\,
C\e^H\rho^{H/2}.
\end{aligned}
$$ 
This proves \eqref{eq:freeze-coeff}.
\end{proof}

 \begin{lemma}\label{lem:remote}
Let $q_H<q\le q_0$, $2\le p\le q/2$,
$0<\rho<1\wedge(T_0/2)$, and
$0<\e\le\sqrt{\rho}$.
Then
$$
\lambda(x)^{1/q}
\left(
\np{J_{\mathrm f}}p +
\np{\sigma(u(t,x))L_{\mathrm f}}p
\right)
\le
C\e\rho^{-(1-H)/2}.
$$
\end{lemma}

\begin{proof}
By the heat-semigroup property and
Lemma~\ref{lem:semigroup-fubini},
\[
J_{\mathrm f}
=
\int_\R D_\rho(z,\e)
\bigl[u(t-\rho,x-z)-(G_{t-\rho}*u_0)(x-z)\bigr]\dd z.
\]
Since $\int_\R D_\rho(z,\e)\dd z=0$, we may write
\begin{equation}\label{eq:remoteconv}
\begin{aligned}
J_{\mathrm f}
=\int_\R D_\rho(z,\e)\bigl\{
&u(t-\rho,x-z)-u(t-\rho,x)\\
&-(G_{t-\rho}*u_0)(x-z)+(G_{t-\rho}*u_0)(x)
\bigr\}\dd z.
\end{aligned}
\end{equation}

By \eqref{eq:umom} and \eqref{eq:exact-holder}, we have 
\[
\np{u(t-\rho,x-z)-u(t-\rho,x)}p
\le
C\bigl[\lambda(x-z)^{-1/q}+\lambda(x)^{-1/q}\bigr]
\left(|z|^H\wedge1\right).
\]

Similarly, using \eqref{eq:msmooth} with $j=0,1$, we obtain
$$
\norm{
(G_{t-\rho}*u_0)(x-z)-(G_{t-\rho}*u_0)(x)
}_p
\le
C\bigl[\lambda(x-z)^{-1/q}+\lambda(x)^{-1/q}\bigr]
(|z|\wedge1).
$$
It follows from \eqref{eq:remoteconv}, Minkowski's inequality,   and
Lemmas~\ref{lem:initial-smoothing}(i) and~\ref{lem:singleenergy}(ii) that
\[
\begin{aligned}
\np{J_{\mathrm f}}p
\le &\,
C\lambda(x)^{-1/q}
\int_\R |D_\rho(z,\e)|
(1+|z|)^{2(1-H)/q}
\left(|z|^H\wedge1\right)\dd z\\
\le &\,
C\lambda(x)^{-1/q}\e\rho^{-(1-H)/2}.
\end{aligned}
\]

By  H\"older's inequality and \eqref{eq:umom}, we have
\[
\lambda(x)^{1/q}
\np{\sigma(u(t,x))L_{\mathrm f}}p
\le
C\np{L_{\mathrm f}}{2p}.
\]
By \eqref{eq:stochastic-isometry}, we have 
 \[
\begin{aligned}
\E|L_{\mathrm f}|^2
&\le
C\int_\rho^\infty\int_\R
e^{-2r\xi^2}|e^{i\e\xi}-1|^2
|\xi|^{1-2H}\dd\xi\dd r\\
&\le
C\e^2\int_\rho^\infty\int_\R
e^{-2r\xi^2}|\xi|^{3-2H}\dd\xi\dd r\\
&=
C\e^2\int_\rho^\infty
r^{H-2}
\left(
\int_\R e^{-2z^2}|z|^{3-2H}\dd z
\right)\dd r\\
&\le
C\e^2\rho^{H-1},
\end{aligned}
\]
where in the third line we used the change of variables
$z=\sqrt r\,\xi$.
Gaussian moment equivalence therefore gives
\[
\lambda(x)^{1/q}
\np{\sigma(u(t,x))L_{\mathrm f}}p
\le
C\e\rho^{-(1-H)/2}.
\]
This proves the lemma.
\end{proof}

\subsection{Proof of Theorem~\ref{thm:local-linearization}(ii)}
\label{sec:spatialbalance}

\begin{proof}Fix $0<T_0<T<\infty$, $q_H<q\le q_0$, and
$0<\delta<\delta_x^*$. We first estimate
$R^{\mathrm x}(t,x,\e)$ for $2\le p\le q/2$ and
$0<\e\le\e_0$. Set $\rho:=\e^{2(1-H)}$ and choose
$\e_0\in(0,1)$ sufficiently small that
$$
 0<\rho<1\wedge(T_0/2),
 \qquad
 \e\le\sqrt{\rho}.
$$ 

For $t\in[T_0,T]$ and $x\in\R$,
\eqref{eq:spatial-u-decomp}--\eqref{eq:Jn-approx} give
\begin{equation}\label{eq:Rx-decomp}
\begin{aligned}
R^{\mathrm x}(t,x,\e)
=&\,
(G_t*u_0)(x+\e)-(G_t*u_0)(x)\\
&+\bigl[J_{\mathrm f}-\sigma(u(t,x))L_{\mathrm f}\bigr]
+\bigl[J_{\mathrm n}-\widetilde J_{\mathrm n}\bigr]\\
&+\bigl[\sigma(u(t-\rho,x))-\sigma(u(t,x))\bigr]
L_{\mathrm n}.
\end{aligned}
\end{equation}
The fundamental theorem of calculus, Minkowski's inequality,
and Lemma~\ref{lem:initial-smoothing}(i)--(ii) yield
$$
\lambda(x)^{1/q}
\norm{(G_t*u_0)(x+\e)-(G_t*u_0)(x)}_p
\le C\e.
$$
Using Lemmas~\ref{lem:freeze} and \ref{lem:remote} in
\eqref{eq:Rx-decomp}, we obtain
$$
\begin{aligned}
\lambda(x)^{1/q}\np{R^{\mathrm x}(t,x,\e)}p
\le&\,
C\Bigl(
\e
+\e\rho^{-(1-H)/2}
+\e^H\rho^{H/2}
+\e^{2H}
+\e^{1/2}
\Bigr)\\
\le&\,
C\Bigl(
\e
+\e^{H+H(1-H)}
+\e^{2H}
+\e^{1/2}
\Bigr),
\end{aligned}
$$
uniformly for $t\in[T_0,T]$ and $x\in\R$.
Since $H>1/4$, we have 
$\e^{2H}\le\e^{1/2}$ for $0<\e<1$.
Moreover,
$$
\delta<\delta_x^*
=
\min\left\{H(1-H),\frac12-H\right\}.
$$
Therefore,
$$
\sup_{T_0\le t\le T}\sup_{x\in\R}
\lambda(x)^{1/q}\np{R^{\mathrm x}(t,x,\e)}p
\le C\e^{H+\delta},
\qquad 0<\e\le\e_0.
$$

The remainder $R^{\mathrm x}(t,x,\e)$ for $\e<0$ is handled similarly by
shifting the base point to $x+\e$.
For $1\le p<2$, the same estimate follows from 
 the monotonicity of $L^p$ norms. The proof is complete. 
\end{proof}

\section{Exact spatial and temporal oscillations}\label{sec:apps}
\subsection{Uniform pathwise
remainders}
  The following lemma is a key ingredient in the proofs of Proposition~\ref{thm:oscillations} and 
Theorem~\ref{thm:uniform-moduli}. Its proof uses a grid interpolation
argument inspired by \cite[Proposition~3.2]{HK2017},
\cite[Lemma~5.1]{Qian2026}, and
\cite[Propositions~4.9--4.10]{HuLee2026}.

\begin{lemma}\label{lemma:pathwise}
Assume Condition~\ref{cond:A} holds with $q_0>q_H$.
\begin{itemize}
\item[(i)] Fix $x\in\R$. With probability one, for every compact interval 
$I\subset(0,\infty)$ and every $0<\delta_t<\delta_{t,\mathrm{pw}}$,
\begin{equation}\label{eq:temporal-pathwise}
 \sup_{\substack{s,t\in I\\|t-s|\le r}}
 |R^{\mathrm t}(s,x,t-s)|=o\left(r^{H/2+\delta_t}\right),\qquad r\downarrow0.
\end{equation}  
\item[(ii)]   Fix $t>0$. With probability one, for every compact interval
$J\subset\R$ and every $0<\delta_x<\delta_{x,\mathrm{pw}}$,
\begin{equation}\label{eq:spatial-pathwise}
 \sup_{\substack{x,y\in J\\|y-x|\le r}}
 |R^{\mathrm x}(t,x,y-x)|=o\left(r^{H+\delta_x}\right),\qquad r\downarrow0.
\end{equation}

\end{itemize}
\end{lemma}
 
\begin{proof}  We prove only part (ii) in detail.  
Part (i) follows by the same grid interpolation argument, using the temporal moment and H\"older estimates. It gives a version of \cite[Lemma~5.1]{Qian2026} that holds uniformly over compact time intervals.

Fix $t>0$, compact intervals $J,\widetilde J\subset\R$ with
$J\subset\operatorname{int}(\widetilde J)$, and
$0<\delta_x<\delta_{x,\mathrm{pw}}$. Set $P:=q_0/2$.
By \eqref{eq:delta-x-pw}, we can choose $m>1$ such that
$$
 \frac{H+\delta_x}{H-1/q_0}
 <m<
 \frac{P(\delta_x^*-\delta_x)+1}{2}.
$$
The strict inequalities allow us to choose
$\delta_x<\delta_1<\delta_2<\delta_x^*$ and
$0<\zeta<H-1/q_0$ such that
\begin{equation}\label{eq:pathwise-parameter-choice}
 m\zeta>H+\delta_x,
 \qquad
 P(\delta_2-\delta_1)+1-2m>0.
\end{equation} 

Since $\lambda$ is bounded away from zero on $\widetilde J$,
Theorem~\ref{thm:local-linearization}(ii), applied with $q=q_0$
and $p=P$, yields
\begin{equation}\label{eq:pathwise-space-moment}
 \|R^{\mathrm x}(t,x,y-x)\|_{P}
 \le C|y-x|^{H+\delta_2}
\end{equation}
for all $x,y\in\widetilde J$ with $|y-x|$ sufficiently small.

  Combining   \eqref{eq:U-increment-Lp} and \eqref{eq:exact-holder} with $p=q=q_0$, and using the compactness of $\widetilde J$, we obtain 
\[
 \mathbb E|u(t,x)-u(t,y)|^{q_0}
 +\mathbb E|U(t,x)-U(t,y)|^{q_0}
 \le C|x-y|^{q_0H},
 \qquad x,y\in \widetilde J.
\]

Since $q_0H>1$, Kolmogorov's continuity theorem
(see, e.g., \cite[p.~10]{DKMNX}) yields continuous versions of
$u(t,\cdot)$ and $U(t,\cdot)$,  whose sample paths are almost surely
$\zeta$-H\"older continuous on $\widetilde J$ for the chosen
$\zeta\in(0,H-1/q_0)$.  Their H\"older continuity and boundedness on $\widetilde J$,
together with the Lipschitz continuity of $\sigma$, imply
\begin{equation}\label{eq:pathwise-space-interpolation}
\begin{aligned}
  |R^{\mathrm x}(t,x,y-x)-R^{\mathrm x}(t,\bar x,\bar y-\bar x)| \le C(\omega)\bigl(|x-\bar x|^\zeta+|y-\bar y|^\zeta\bigr),
 \quad x,y,\bar x,\bar y\in \widetilde J,
\end{aligned}
\end{equation}
where $C(\omega)<\infty$ almost surely.

 Let $r_n=2^{-n}$, $a_n=r_n^m$, and
$$
 \mathcal G_n=(a_n\mathbb Z)^2\cap(\widetilde J\times\widetilde J).
$$
The interval $\widetilde J$ contains $O(a_n^{-1})$ grid points $x$. For each such $x$, the interval $[x-2r_n,x+2r_n]$ contains
$O(r_n/a_n)$ grid points $y$.  Hence,
\begin{equation}\label{eq:pathwise-grid-count}
 \#\{\,(x,y)\in\mathcal G_n:|y-x|\le2r_n\,\}
 \le Cr_na_n^{-2}=Cr_n^{1-2m}.
\end{equation}

 For all sufficiently large $n$,  by Markov's
inequality,   \eqref{eq:pathwise-space-moment},
and \eqref{eq:pathwise-grid-count},  we have 
\begin{align*}
 \PP\left\{
 \max_{\substack{(x,y)\in\mathcal G_n\\|y-x|\le2r_n}}
 |R^{\mathrm x}(t,x,y-x)|>r_n^{H+\delta_1}
 \right\} 
 \le &\,
 r_n^{-P(H+\delta_1)}
 \sum_{\substack{(x,y)\in\mathcal G_n\\|y-x|\le2r_n}}
 \|R^{\mathrm x}(t,x,y-x)\|_{P}^P\\
 \le &\,  Cr_n^{P(\delta_2-\delta_1)+1-2m}.
\end{align*}
By \eqref{eq:pathwise-parameter-choice}, the exponent in the last term 
is positive, so these probabilities are summable. The Borel--Cantelli 
lemma therefore yields, almost surely, for all sufficiently large $n$,
\begin{equation}\label{eq:pathwise-grid-bound}
 \max_{\substack{(x,y)\in\mathcal G_n\\|y-x|\le2r_n}}
 |R^{\mathrm x}(t,x,y-x)|\le r_n^{H+\delta_1}.
\end{equation} 

Since $J\subset\operatorname{int}(\widetilde J)$ and $a_n=o(r_n)$,
for all sufficiently large $n$ and all $x,y\in J$ with
$|y-x|\le r_n$, we can choose $(\bar x,\bar y)\in\mathcal G_n$
such that
$$
 |x-\bar x|+|y-\bar y|\le2a_n,
 \qquad
 |\bar y-\bar x|\le r_n+2a_n\le2r_n.
$$
 
Combining \eqref{eq:pathwise-space-interpolation} and
\eqref{eq:pathwise-grid-bound} yields, almost surely,
\[
 \sup_{\substack{x,y\in J\\|y-x|\le r_n}}
 |R^{\mathrm x}(t,x,y-x)|
 \le C(\omega)r_n^{m\zeta}+r_n^{H+\delta_1}
 =o(r_n^{H+\delta_x}).
\]
Since the supremum is nondecreasing in $r$ and $r_n/r_{n+1}=2$,
the estimate extends to arbitrary $r\downarrow0$, proving
\eqref{eq:spatial-pathwise} for the  chosen $J$ and $\delta_x$.

Taking the intersection over all $k\ge1$ and
$\gamma\in\mathbb Q\cap(0,\delta_{x,\mathrm{pw}})$,
we obtain a single event of probability one on which
\eqref{eq:spatial-pathwise} holds with $J=J_k:=[-k,k]$
and $\delta_x=\gamma$ for every such pair $(k,\gamma)$.
Now fix any compact interval $J\subset\R$ and
$0<\delta_x<\delta_{x,\mathrm{pw}}$.
Choose $k\ge1$ such that $J\subset J_k$ and
a rational number $\gamma$ satisfying
$\delta_x<\gamma<\delta_{x,\mathrm{pw}}$.
On the same probability-one event, the estimate on $J_k$
with exponent $\gamma$ also holds on $J$.
Since
\[
r^{H+\gamma}=o(r^{H+\delta_x})
\qquad\text{as }r\downarrow0,
\]
the desired assertion follows simultaneously for all
such $J$ and $\delta_x$.

 The proof is complete. 
 \end{proof}

\subsection{Proof of Proposition~\ref{thm:oscillations}}
\label{sec:gaussian-inputs}
 We recall two decompositions of $U$ in terms of fBm and smooth Gaussian corrections.

 \begin{lemma}\label{lem:completion}  
\begin{itemize}
\item[(i)] (\cite[Theorem~1(b)]{WangZhang2021}) For each fixed $t>0$, there exist a standard two-sided fBm
$\{B_t^H(y):y\in\R\}$ and a centered Gaussian process
$\{S_t(y):y\in\R\}$, with $S_t(0)=0$ and
$C^\infty(\R)$ sample paths, such that 
\begin{equation}\label{eq:completion}
 U(t,y)-U(t,0)
 =\frac1{\sqrt2}B_t^H(y)-S_t(y),
 \qquad y\in\R.
\end{equation}

\item[(ii)] (\cite[Lemma~6.1]{Qian2026})
For each fixed $x\in\R$, there exists a centered Gaussian process
$\{S(t): t\ge0\}$, independent of $\{U(t,x): t\ge0\}$, such that the process
\begin{equation}\label{eq:completion-time}
 B^{H/2}(t):=\gamma_H^{-1}[U(t,x)+S(t)],
 \qquad t\ge0,
\end{equation}
is a standard fBm. Moreover, $S$ admits a version with
$C^\infty(0,\infty)$ sample paths.
\end{itemize}
\end{lemma}

We recall Khinchin's and Chung's LILs,
together with the L\'evy-type exact uniform modulus of continuity
for fBm; see, e.g., \cite[Section~7]{LiShao2001},
\cite[Theorem~3.3]{MR1995}, and \cite[Theorem~1.2]{Tal96}.

\begin{lemma} \label{lem:fbm-laws} Let $B^\alpha$ be a standard fBm with $0<\alpha<1$,
defined on either $[0,\infty)$ or $\R$. 
\begin{itemize}
\item[(i)] \emph{(Khinchin-type LIL).}
For each fixed $z_0$ in the parameter domain of $B^\alpha$,
 
\begin{equation}\label{eq:fbm-khinchin}
 \limsup_{h\downarrow0}
 \frac{B^\alpha(z_0+h)-B^\alpha(z_0)}
      {h^\alpha\sqrt{2\log\log(1/h)}}
 =1 \ \ \ \text{a.s.}
\end{equation}

\item[(ii)] \emph{(Chung-type  LIL).} 
For each fixed $z_0$ in the parameter domain of $B^\alpha$,
\begin{equation}\label{eq:fbm-chung}
 \liminf_{r\downarrow0}\sup_{0\le h\le r}
 \frac{ 
       |B^\alpha(z_0+h)-B^\alpha(z_0)|}
      {(r/\log\log(1/r))^\alpha}
 =b_\alpha^\alpha  \ \ \ \text{a.s.,}
\end{equation}
where $b_\alpha$ is defined in \eqref{eq:fbm-small-constant}.

  \item[(iii)] \emph{(Exact uniform modulus of continuity).}
For each nondegenerate compact interval $I$, 
 \[
 \lim_{r\downarrow0}\sup_{\substack{s,t\in I\\ |t-s|\le r}}
 \frac{          |B^\alpha(t)-B^\alpha(s)|}
      {r^\alpha\sqrt{2\log(1/r)}}
 =1   \ \ \ \text{a.s.}
\]
\end{itemize}
\end{lemma}
 \begin{proof}[Proof of Proposition~\ref{thm:oscillations}]
Fix $(t,x)\in(0,\infty)\times\R$. We follow the argument
in~\cite{FKM2015,WangXiao2024}.
Let $\alpha=H$ for spatial increments and $\alpha=H/2$ for
temporal increments. By Lemmas~\ref{lemma:pathwise} and
\ref{lem:completion}, the remainders and smooth correction terms
are negligible relative to
$$
 r^\alpha\sqrt{2\log\log(1/r)},
 \qquad
 \left(\frac{r}{\log\log(1/r)}\right)^\alpha,
$$
the normalizing factors in Lemma~\ref{lem:fbm-laws}(i)--(ii),
respectively.   The assertions then follow.
\end{proof}

\begin{remark}\label{rem:normalized-moduli}
Lemma~\ref{lemma:pathwise}, together with
Lemmas~\ref{lem:completion} and \ref{lem:fbm-laws}(iii), also yields
the following coefficient-normalized uniform moduli of continuity,
analogous to those obtained by Hu and
Lee~\cite[Theorem~1.2]{HuLee2026}. More precisely,
if $\delta_{x,\mathrm{pw}}>0$, then for each fixed $t>0$, 
every nondegenerate compact interval $J\subset\R$ with
$$
 \inf_{x\in J}|\sigma(u(t,x))|>0
$$
satisfies
\begin{align}\label{eq Lee1}
 \lim_{r\downarrow0}
 \sup_{\substack{x,y\in J\\ |y-x|\le r}}
 \frac{|u(t,y)-u(t,x)|}
 {r^H\sqrt{2\log(1/r)}\,|\sigma(u(t,x))|}
 =\frac1{\sqrt2} \ \ \ \text{a.s.}
\end{align}
Similarly, if $\delta_{t,\mathrm{pw}}>0$, then for each fixed $x\in\R$,  every nondegenerate compact interval
$I\subset(0,\infty)$ with
$$
 \inf_{s\in I}|\sigma(u(s,x))|>0
$$
satisfies
\begin{align}\label{eq Lee2}
 \lim_{r\downarrow0}
 \sup_{\substack{s,t\in I\\ |t-s|\le r}}
 \frac{|u(t,x)-u(s,x)|}
 {r^{H/2}\sqrt{2\log(1/r)}\,|\sigma(u(s,x))|}
 =\gamma_H \ \ \ \text{a.s.}
\end{align}

Since $|\sigma(u)|$ is bounded away from zero on each such interval,
the remainders and smooth Gaussian corrections remain uniformly
negligible under the corresponding normalizations.
Both assertions therefore follow from the exact uniform moduli
of continuity of fBm.
The lower bound in Theorem~\ref{thm:uniform-moduli}, however,
requires an additional localization argument near points where
$|\sigma(u)|$ is close to its maximum on the compact interval.
We present this argument in the next subsection.

\end{remark}

\subsection{Proof of Theorem~\ref{thm:uniform-moduli}}\label{sec:spatial-modulus}

\begin{proof}[Proof of Theorem~\ref{thm:uniform-moduli}]
For the spatial assertion \eqref{eq:spatial-modulus}, assume
$\delta_{x,\mathrm{pw}}>0$ and fix $t>0$.
Let
\[
\mathcal Q=\{[a,b]:a,b\in\mathbb Q,\ a<b\}.
\]

\noindent\emph{Step 1: A common probability-one event.}
Let $B_t^H$ and $S_t$ be the processes in
Lemma~\ref{lem:completion}(i). By Lemma~\ref{lem:fbm-laws}(iii)
and the countability of $\mathcal Q$, there is a probability-one
event on which
\begin{align}\label{eq LILsim}
\lim_{r\downarrow0}
\sup_{\substack{x,y\in I\\ |y-x|\le r}}
\frac{|B_t^H(y)-B_t^H(x)|}{r^H\sqrt{2\log(1/r)}}
=1
\qquad\text{for every }I\in\mathcal Q.
\end{align}

 Work on a common probability-one event $\Omega_t^{\mathrm{sp}}$
on which \eqref{eq LILsim} and the conclusion of  Lemma~\ref{lemma:pathwise}(ii) hold,
$S_t$ is smooth, and $\sigma(u(t,\cdot))$ is continuous.
For an arbitrary nondegenerate compact interval
$J=[a,b]\subset\R$, set
\[
M_J:=\max_{x\in J}|\sigma(u(t,x))|.
\]

   \noindent\emph{Step 2: Upper bound.} Fix $0<\delta_x<\delta_{x,\mathrm{pw}}$.
By \eqref{eq Rx} and Lemma~\ref{lem:completion}(i),
for $x,y\in J$ with $|y-x|\le r$, we have 
\[
\begin{aligned}
&\left|
u(t,y)-u(t,x)
-\frac{\sigma(u(t,x))}{\sqrt2}
\bigl(B_t^H(y)-B_t^H(x)\bigr)
\right|\\
 \le&\, 
|R^{\mathrm x}(t,x,y-x)|
+
|\sigma(u(t,x))|\,|S_t(y)-S_t(x)|\\
 \le &\, 
 |R^{\mathrm x}(t,x,y-x)|
+
M_J\sup_{z\in J}|S_t'(z)|\,r.
\end{aligned}
\]
Here we used the smoothness of $S_t$ on $\R$ and the fact that
$J$ is compact. Define
\[
d_J(r)
:=
\sup_{\substack{x,y\in J\\ |y-x|\le r}}
|R^{\mathrm x}(t,x,y-x)|
+
M_J\sup_{z\in J}|S_t'(z)|\,r.
\]

Since $H<1$ and $\delta_x>0$, \eqref{eq:spatial-pathwise} gives 
\begin{align}\label{eq dj}
\lim_{r\downarrow 0}\frac{d_J(r)}{r^H\sqrt{2\log(1/r)}}
=0, \ \ \text{ a.s.}
\end{align}

Let
\[
\omega_{t,J}(r)
:=
\sup_{\substack{x,y\in J\\ |y-x|\le r}}
|u(t,y)-u(t,x)|.
\]
Taking the supremum in the above estimate gives
\[
\omega_{t,J}(r)
\le
\frac{M_J}{\sqrt2}
\sup_{\substack{x,y\in J\\ |y-x|\le r}}
|B_t^H(y)-B_t^H(x)|
+d_J(r).
\] 
Choose $K\in\mathcal Q$ containing $J$. By \eqref{eq LILsim} and \eqref{eq dj}, we have
\[
\limsup_{r\downarrow0}
\frac{\omega_{t,J}(r)}{r^H\sqrt{2\log(1/r)}}
\le \frac{M_J}{\sqrt2}.
\]

\noindent\emph{Step 3: Lower bound.}
If $M_J=0$, the conclusion follows immediately from the upper bound.
Assume now that $M_J>0$ and fix $0<\eta<M_J$.

By the continuity of $\sigma(u(t,\cdot))$, the set
\[
\{x\in(a,b): |\sigma(u(t,x))|>M_J-\eta\}
\]
is a nonempty open subset of $(a,b)$. Hence, it contains a
nondegenerate compact interval
$I_\eta\subset J$ with rational endpoints such that
\[
|\sigma(u(t,x))|\ge M_J-\eta,
\qquad x\in I_\eta.
\]
Thus $I_\eta\in\mathcal Q$.

For $x,y\in I_\eta$ with $|y-x|\le r$, the estimate from
Step~2 gives
\[
|u(t,y)-u(t,x)|
\ge
\frac{M_J-\eta}{\sqrt2}
|B_t^H(y)-B_t^H(x)|
-d_J(r).
\]
By \eqref{eq LILsim} and \eqref{eq dj},
\[
\liminf_{r\downarrow0}
\frac{\omega_{t,J}(r)}{r^H\sqrt{2\log(1/r)}}
\ge
\frac{M_J-\eta}{\sqrt2}.
\]
Letting $\eta\downarrow0$ and combining this lower bound with
the upper bound proves \eqref{eq:spatial-modulus} simultaneously for
every nondegenerate compact interval $J$ on $\Omega_t^{\mathrm{sp}}$.

For the temporal assertion, fix $x\in\R$ and assume
$\delta_{t,\mathrm{pw}}>0$.  The same argument applies with
Lemmas \ref{lemma:pathwise}(i), \ref{lem:completion}(ii), and \ref{lem:fbm-laws}(iii),   replacing $H$ and
$1/\sqrt{2}$ by $H/2$ and $\gamma_H$, respectively.
Using compact intervals with rational endpoints in $(0,\infty)$
to construct the common probability-one event yields
\eqref{eq:temporal-modulus} simultaneously for every nondegenerate
compact interval $I\subset(0,\infty)$.

The proof is complete. 
\end{proof}

\appendix
\section{Auxiliary heat-kernel estimates}\label{app:kernels}
\subsection{A fractional-difference estimate for Schwartz functions}

The following estimate extends \cite[Lemma~2.10]{HW2022}
from the heat kernel to Schwartz functions. We will use it
for derivatives of the heat kernel.

\begin{lemma}\label{lem:schwartz}
Let $0<H<1/2$. For every Schwartz function $f$ on $\R$,
\begin{equation}\label{eq:schwartz}
Q_Hf(z)
:=
\int_\R |f(z+h)-f(z)|^2|h|^{2H-2}\dd h
\le
C_{f,H}(1+|z|)^{2H-2},
\quad z\in\R.
\end{equation}
\end{lemma}

\begin{proof} For $|z|\le2$, we apply the mean-value theorem when $|h|\le1$
and use the boundedness of $f$ when $|h|>1$ to obtain 
$$
\begin{aligned}
Q_Hf(z)
\le&\,
\norm{f'}_\infty^2
\int_{|h|\le1}|h|^{2H}\dd h
+
4\norm f_\infty^2
\int_{|h|>1}|h|^{2H-2}\dd h\\
\le&\, C_{f,H}.
\end{aligned}
$$

Now let $|z|>2$ and split the integral at $|h|=|z|/2$.
For $|h|\le|z|/2$, the rapid decay of $f$ and $f'$ yields
$$
|f(z+h)-f(z)|
\le
C_f(1+|z|)^{-1}(|h|\wedge1),
$$
and hence
$$
\int_{|h|\le|z|/2}
|f(z+h)-f(z)|^2|h|^{2H-2}\dd h
\le
C_{f,H}(1+|z|)^{-2}.
$$
On the complementary region, using
$|a-b|^2\le2|a|^2+2|b|^2$, we obtain 
$$
\begin{aligned}
 \int_{|h|>|z|/2}
|f(z+h)-f(z)|^2|h|^{2H-2}\dd h 
\le&\,
C_H|z|^{2H-2}\norm f_{L^2(\R)}^2
+
C_H|f(z)|^2|z|^{2H-1}\\
\le&\,
C_{f,H}(1+|z|)^{2H-2}.
\end{aligned}
$$
Since
$$
(1+|z|)^{-2}\le(1+|z|)^{2H-2},
$$
the desired estimate follows.
\end{proof}

\subsection{Weighted heat-kernel estimates for temporal increments}\label{sec:temporal-kernels}

 Recall the heat-kernel increment $D_r$ defined in \eqref{eq:Dr}, and the constant $a_q$ defined in \eqref{eq a}.

\begin{lemma}\label{t:lem:kernels}
Let $q_H<q\le q_0$, $0<r\le T$, $0<\e<1$, and
$$
0\le\vartheta<\frac{1-2H-a_q}{4}.
$$ 
Then, uniformly in \(x\in\R\), 
 \begin{align}
\int_\R G_r(x-y)^2\lambda(y)^{-2/q}\dd y
\le  &\, C\lambda(x)^{-2/q}r^{-1/2},\label{t:eq:kernel-L2}\\ 
\int_{\R^2}|D_r(x-y,h)|^2|h|^{2H-2}|x-y|^{4\vartheta}\lambda(y)^{-2/q}\dd y\dd h
\le  &\, C\lambda(x)^{-2/q}r^{H+2\vartheta-1},\label{t:eq:kernel-weighted}\\
\int_\R |G_{r+\e}(x-y)-G_r(x-y)|^2\lambda(y)^{-2/q}\dd y
\le  &\, C\lambda(x)^{-2/q}\e^2r^{-5/2},\label{t:eq:kernel-time-L2}\\
\int_{\R^2}|D_{r+\e}(x-y,h)-D_r(x-y,h)|^2|h|^{2H-2}\lambda(y)^{-2/q}\dd y\dd h
\le  &\, C\lambda(x)^{-2/q}\e^2r^{H-3}.\label{t:eq:kernel-time-H}
\end{align}
 
\end{lemma}

 \begin{proof}
The scaling identities
\[
G_r(z)=r^{-1/2}G_1(z/\sqrt r),
\qquad
D_r(z,h)=r^{-1/2}D_1(z/\sqrt r,h/\sqrt r),
\]
together with the substitutions
$y=x-\sqrt r\,z$, $h=\sqrt r\,w$, and
Lemma~\ref{lem:initial-smoothing}(i), reduce
\eqref{t:eq:kernel-L2} and \eqref{t:eq:kernel-weighted}
to the finiteness of
\[
\int_\R G_1(z)^2(1+|z|)^{a_q}\dd z
\]
and
\[
\int_{\R^2}
|D_1(z,w)|^2|w|^{2H-2}|z|^{4\vartheta}
(1+|z|)^{a_q}\dd z\dd w,
\]
respectively.
The first integral is finite by Gaussian decay, while
Lemma~\ref{lem:schwartz} yields
\[
\int_{\R^2}
|D_1(z,w)|^2|w|^{2H-2}|z|^{4\vartheta}
(1+|z|)^{a_q}\dd z\dd w
\le
C\int_\R
|z|^{4\vartheta}(1+|z|)^{a_q+2H-2}\dd z<\infty,
\]
since $a_q+4\vartheta<1-2H$.

Since
\[
\partial_rG_r(z)
=
r^{-3/2}
\left(\frac{z^2}{4r}-\frac12\right)
G_1(z/\sqrt r),
\]
and $(w^2/4-1/2)G_1(w)$ is a Schwartz function,
the same scaling arguments give
\begin{align}
\int_\R|\partial_rG_r(x-y)|^2
\lambda(y)^{-2/q}\dd y
&\le
C\lambda(x)^{-2/q}r^{-5/2},
\label{eq:derivative-temporal}\\
\int_{\R^2}
|\partial_rG_r(x+h-y)-\partial_rG_r(x-y)|^2
|h|^{2H-2}\lambda(y)^{-2/q}
\dd h\dd y
&\le
C\lambda(x)^{-2/q}r^{H-3}.
\label{eq:derivative-spatial}
\end{align}
Since $r\le T$ and $0<\e<1$, these estimates remain valid,
with a uniform constant, when $r$ is replaced by $r+s$
for $0\le s\le\e$.

The identity
$$
G_{r+\e}-G_r
=
\int_0^\e \partial_sG_{r+s}\dd s,
$$
together with Minkowski's inequality and
\eqref{eq:derivative-temporal}, gives 
\[
\begin{aligned}
&\int_\R
|G_{r+\e}(x-y)-G_r(x-y)|^2
\lambda(y)^{-2/q}\dd y\\
\le&\,
\left(
\int_0^\e
\left(
\int_\R
|\partial_sG_{r+s}(x-y)|^2
\lambda(y)^{-2/q}\dd y
\right)^{1/2}
\dd s
\right)^2\\
\le&\,
C\lambda(x)^{-2/q}
\left(
\int_0^\e(r+s)^{-5/4}\dd s
\right)^2\\
\le&\,
C\lambda(x)^{-2/q}\e^2r^{-5/2}.
\end{aligned}
\]

Similarly,
\[
D_{r+\e}(x,h)-D_r(x,h)
=
\int_0^\e\partial_sD_{r+s}(x,h)\dd s.
\]
Hence, by Minkowski's inequality and
\eqref{eq:derivative-spatial},
\[
\begin{aligned}
&\int_{\R^2}
|D_{r+\e}(x-y,h)-D_r(x-y,h)|^2
|h|^{2H-2}\lambda(y)^{-2/q}\dd y\dd h\\
\le&\,
\left(
\int_0^\e
\left(
\int_{\R^2}
|\partial_sD_{r+s}(x-y,h)|^2
|h|^{2H-2}\lambda(y)^{-2/q}
\dd y\dd h
\right)^{1/2}
\dd s
\right)^2\\
\le&\,
C\lambda(x)^{-2/q}
\left(
\int_0^\e(r+s)^{(H-3)/2}\dd s
\right)^2\\
\le&\,
C\lambda(x)^{-2/q}\e^2r^{H-3}.
\end{aligned}
\]
The proof is complete.
\end{proof}

 The following integrated estimates are immediate consequences of
Lemma~\ref{t:lem:kernels} and will be used in the proof of the
temporal local linearization.

\begin{lemma}\label{t:lem:integrated}
Let $q_H<q\le q_0$ and $0<\beta<1$.  There exists a constant $C>0$ such that the following estimates
hold uniformly for $x\in\R$ and all sufficiently small
$\varepsilon>0$.  
\begin{itemize}
\item[(i)]
\begin{equation}\label{t:eq:integrated-remote}
\begin{aligned}
&\int_{\e^\beta}^T\int_{\R^2}
|D_{r+\e}(x-y,h)-D_r(x-y,h)|^2
|h|^{2H-2}\lambda(y)^{-2/q}
\dd y\dd h\dd r\\
\le&\,
C\lambda(x)^{-2/q}\e^{2+\beta(H-2)}.
\end{aligned}
\end{equation}

\item[(ii)]
\begin{equation}\label{t:eq:integrated-L2}
\int_0^T\int_\R
|G_{r+\e}(x-y)-G_r(x-y)|^2
\lambda(y)^{-2/q}\dd y\dd r
\le
C\lambda(x)^{-2/q}\e^{1/2}.
\end{equation}

\item[(iii)]
If
$$
0\le\vartheta<\frac{1-2H-a_q}{4},
$$
then
\begin{equation}\label{t:eq:integrated-weighted}
\begin{aligned}
&\int_0^{\e^\beta}\int_{\R^2}
\bigl(
|D_r(x-y,h)|^2
+
|D_{r+\e}(x-y,h)|^2
\bigr)
|h|^{2H-2}|x-y|^{4\vartheta}
\lambda(y)^{-2/q}
\dd y\dd h\dd r\\
\le&\,
C\lambda(x)^{-2/q}\e^{\beta(H+2\vartheta)}.
\end{aligned}
\end{equation}

\end{itemize}
\end{lemma}
\begin{proof}
For sufficiently small $\e>0$, we may assume that
$\e^\beta<T$.

For (i), integrating \eqref{t:eq:kernel-time-H} over
$r\in[\e^\beta,T]$ gives
$$
\begin{aligned}
&\int_{\e^\beta}^T\int_{\R^2}
|D_{r+\e}(x-y,h)-D_r(x-y,h)|^2
|h|^{2H-2}\lambda(y)^{-2/q}\dd y\dd h\dd r\\
\le&\,
C\lambda(x)^{-2/q}\e^2
\int_{\e^\beta}^T r^{H-3}\dd r\\
\le&\,
C\lambda(x)^{-2/q}\e^{2+\beta(H-2)}.
\end{aligned}
$$

For (ii), split the integral over $(0,\e)$ and $(\e,T)$.
On $(0,\e)$, using \eqref{t:eq:kernel-L2} and
$$
|G_{r+\e}-G_r|^2
\le
2|G_{r+\e}|^2+2|G_r|^2,
$$
we obtain
$$
\begin{aligned}
&\int_0^\e\int_\R
|G_{r+\e}(x-y)-G_r(x-y)|^2
\lambda(y)^{-2/q}\dd y\dd r\\
\le&\,
C\lambda(x)^{-2/q}
\int_0^\e
\bigl((r+\e)^{-1/2}+r^{-1/2}\bigr)\dd r\\
\le&\,
C\lambda(x)^{-2/q}\e^{1/2}.
\end{aligned}
$$
On $(\e,T)$, \eqref{t:eq:kernel-time-L2} gives
$$
\begin{aligned}
 \int_\e^T\int_\R
|G_{r+\e}(x-y)-G_r(x-y)|^2
\lambda(y)^{-2/q}\dd y\dd r 
\le&\,
C\lambda(x)^{-2/q}\e^2
\int_\e^T r^{-5/2}\dd r\\
\le&\,
C\lambda(x)^{-2/q}\e^{1/2}.
\end{aligned}
$$
This proves \eqref{t:eq:integrated-L2}.

For (iii), \eqref{t:eq:kernel-weighted} yields
$$
\begin{aligned}
&\int_0^{\e^\beta}\int_{\R^2}
\bigl(
|D_r(x-y,h)|^2+
|D_{r+\e}(x-y,h)|^2
\bigr)
|h|^{2H-2}|x-y|^{4\vartheta}
\lambda(y)^{-2/q}\dd y\dd h\dd r\\
\le&\,
C\lambda(x)^{-2/q}
\int_0^{\e^\beta}
\bigl(
r^{H+2\vartheta-1}
+(r+\e)^{H+2\vartheta-1}
\bigr)\dd r.
\end{aligned}
$$
Using $\e\le\e^\beta$ and $H+2\vartheta>0$, we obtain
$$
\int_0^{\e^\beta}
\bigl(
r^{H+2\vartheta-1}
+(r+\e)^{H+2\vartheta-1}
\bigr)\dd r
\le
C\e^{\beta(H+2\vartheta)}.
$$
Thus \eqref{t:eq:integrated-weighted} follows.
\end{proof}

 \subsection{Spatial fractional-difference estimates}
\label{app:spatial-kernels}

The following energy identity is implicit in the proof of
\cite[Theorem~1(b)]{WangZhang2021}. We record it here in terms of
the heat-kernel increment $D_r$.

\begin{lemma}\label{lem:energy}
For every $\e>0$,
\begin{equation}\label{eq:energy}
\int_0^\infty\int_{\R^2}
|D_r(z+h,\e)-D_r(z,\e)|^2|h|^{2H-2}
\dd h\dd z\dd r
=
\frac{\e^{2H}}{H(1-2H)}.
\end{equation}
\end{lemma}
 \begin{proof}
By \eqref{eq:inner},
$$
\begin{aligned}
&\int_0^\infty\int_{\R^2}
|D_r(z+h,\e)-D_r(z,\e)|^2
|h|^{2H-2}\dd h\dd z\dd r\\
=&\,
\frac{c_{1,1}}{H(\frac12-H)}
\int_0^\infty\int_\R
|e^{i\e\xi}-1|^2e^{-2r\xi^2}
|\xi|^{1-2H}\dd\xi\dd r\\
=&\,
\frac{c_{1,1}}{2H(\frac12-H)}
\int_\R
|e^{i\e\xi}-1|^2|\xi|^{-1-2H}\dd\xi\\
=&\,
\frac{c_{1,1}\e^{2H}}{2H(\frac12-H)}
\int_\R
|e^{iu}-1|^2|u|^{-1-2H}\dd u\\
=&\,
\frac{\e^{2H}}{H(1-2H)}.
\end{aligned}
$$   Here we used the change of variables $u=\e\xi$ and the identity
$$
 \int_\R |e^{iu}-1|^2|u|^{-1-2H}\dd u
 =\frac{2\pi}{\Gamma(2H+1)\sin(\pi H)}
 =c_{1,1}^{-1}.
$$

The proof is complete. 
\end{proof} 

With $Q_H$ as defined in \eqref{eq:schwartz}, set
$$
F_H(w):=\int_0^\infty Q_H[D_s(\cdot,1)](w)\dd s.
$$

\begin{lemma}\label{lem:tails}
Let $0<\gamma<1-H$ and $R\ge1$. The following estimates hold.
\begin{itemize}

\item[(i)]
\begin{equation}\label{eq:localmoment}
\int_{|w|\le R}|w|^{2\gamma}F_H(w)\dd w
\le
C
\begin{cases}
1+R^{2\gamma+2H-1},
& 2\gamma+2H-1>0,\\
1+\log R,
& 2\gamma+2H-1=0,\\
1,
& 2\gamma+2H-1<0.
\end{cases}
\end{equation}

\item[(ii)]
For every $0\le a<1-2H$,
\begin{equation}\label{eq:weightedtail}
\int_{|w|>R}(1+|w|)^aF_H(w)\dd w
\le
CR^{a+2H-1}.
\end{equation}

\end{itemize}
\end{lemma}

\begin{proof}
Lemma~\ref{lem:energy} with \(\e=1\) yields
\begin{align}\label{eq F-integ}
     \int_\R F_H(w)\dd w<\infty.
\end{align}

We first show that
\begin{equation}\label{eq:FH-pointwise}
 F_H(w)\le C|w|^{2H-2},
 \qquad |w|>2.
\end{equation}
For \(0<s\le1\),
\[
 Q_H[D_s(\cdot,1)](w)
 \le 2Q_HG_s(w+1)+2Q_HG_s(w).
\]
By the scaling identity
\[
 Q_HG_s(w)
 =
 s^{H-3/2}Q_HG_1(w/\sqrt s)
\]
and Lemma~\ref{lem:schwartz}, we obtain, for \(|w|>2\),
\[
 Q_H[D_s(\cdot,1)](w)
 \le C|w|^{2H-2}s^{-1/2}.
\]
Consequently,
\[
 \int_0^1 Q_H[D_s(\cdot,1)](w)\dd s
 \le C|w|^{2H-2}.
\]

We also note the scaling relation
\begin{equation}\label{eq:heat-kernel-derivative-scaling}
G_s'(y)=s^{-1}G_1'(y/\sqrt{s}),
\qquad s>0,\ y\in\R.
\end{equation}
For $s\ge1$, we have
$$
D_s(\cdot,1)=\int_0^1G_s'(\cdot+v)\,\dd v.
$$
Jensen's inequality, \eqref{eq:heat-kernel-derivative-scaling},
and Lemma~\ref{lem:schwartz} therefore yield  
$$
Q_H[D_s(\cdot,1)](w)
\le
Cs^{H-5/2}
\left(1+\frac{|w|}{\sqrt{s}}\right)^{2H-2}.
$$
 
Hence, for \(|w|>2\),
\[
\begin{aligned}
\int_1^\infty Q_H[D_s(\cdot,1)](w)\dd s
\le &\, 
C|w|^{2H-2}
\int_1^{|w|^2/4}s^{-3/2}\dd s
+
C\int_{|w|^2/4}^\infty s^{H-5/2}\dd s\\
\le &\, 
C|w|^{2H-2}.
\end{aligned}
\]
This proves \eqref{eq:FH-pointwise}.

For (i),  \eqref{eq F-integ} and
\eqref{eq:FH-pointwise} imply, for \(R>2\),
\[
\begin{aligned}
\int_{|w|\le R}|w|^{2\gamma}F_H(w)\dd w
\le &\, 
C+
C\int_2^R r^{2\gamma+2H-2}\dd r.
\end{aligned}
\]
The three cases in \eqref{eq:localmoment}  follow by considering whether
\(2\gamma+2H-1\) is positive, zero, or negative.
The case \(1\le R\le2\) follows directly from  \eqref{eq F-integ}.

For (ii), when \(R\ge2\),   by \eqref{eq:FH-pointwise} and \(a<1-2H\),  we have 
\[
\begin{aligned}
\int_{|w|>R}(1+|w|)^aF_H(w)\dd w
\le &\, 
C\int_R^\infty r^{a+2H-2}\dd r\\
\le &\, 
CR^{a+2H-1}.
\end{aligned}
\] 

The case $1\le R<2$ follows from \eqref{eq F-integ}
and the estimate above with $R=2$, after enlarging $C$.
This proves \eqref{eq:weightedtail}.
\end{proof}

 \begin{lemma}\label{lem:weightedenergy}
Let $0\le a<1-2H$ and $0<\e\le1$. There exists a constant
$C>0$ such that
\begin{equation}\label{eq:wenergy}
\int_0^\infty\int_{\R^2}
|D_r(z+h,\e)-D_r(z,\e)|^2
|h|^{2H-2}(1+|z|)^a
\dd h\dd z\dd r
\le
C\e^{2H}.
\end{equation}
Moreover, for every $0<\gamma<1-H$,
\begin{equation}\label{eq:wlocalenergy}
\begin{aligned}
&\int_0^\infty\int_{\R^2}
|D_r(z+h,\e)-D_r(z,\e)|^2
|h|^{2H-2}(1+|z|)^a
(|z|^{2\gamma}\wedge1)
\dd h\dd z\dd r\\
\le&\,
C
\begin{cases}
\e^{2H+2\gamma}+\e,
& \gamma\ne\frac12-H,\\[1mm]
\e(1+|\log\e|),
& \gamma=\frac12-H.
\end{cases}
\end{aligned}
\end{equation}
\end{lemma}

  \begin{proof}
With the change of variables
$$
r=\e^2s,\qquad z=\e w,\qquad h=\e\ell,
$$
the scaling of $D_r$ gives
$$
D_{\e^2s}(\e w+\e\ell,\e)-D_{\e^2s}(\e w,\e)
=
\e^{-1}\bigl[D_s(w+\ell,1)-D_s(w,1)\bigr].
$$
Hence, by Fubini's theorem, for every nonnegative measurable function $f$,
\begin{equation}\label{eq:scaling}
\int_0^\infty\int_{\R^2}
|D_r(z+h,\e)-D_r(z,\e)|^2
|h|^{2H-2}f(z)\dd h\dd z\dd r
=
\e^{2H}\int_\R F_H(w)f(\e w)\dd w.
\end{equation}

Taking $f(z)=(1+|z|)^a$ in \eqref{eq:scaling} and using
$0<\e\le1$, we obtain
$$
\int_\R F_H(w)(1+\e|w|)^a\dd w
\le
\int_\R F_H(w)(1+|w|)^a\dd w.
$$
The integral on the right is finite: on $|w|\le1$ this follows from
$F_H\in L^1(\R)$, while on $|w|>1$ it follows from
\eqref{eq:weightedtail} with $R=1$. This proves
\eqref{eq:wenergy}.

For \eqref{eq:wlocalenergy}, take
$$
f(z)=(1+|z|)^a(|z|^{2\gamma}\wedge1)
$$
in \eqref{eq:scaling}, and write
$$
\e^{2H}\int_\R
F_H(w)(1+\e|w|)^a
(\e^{2\gamma}|w|^{2\gamma}\wedge1)\dd w
=:J_1+J_2,
$$
where $J_1$ and $J_2$ are the integrals over
$|w|\le\e^{-1}$ and $|w|>\e^{-1}$, respectively.

For $J_1$,
$$
J_1
\le
C\e^{2H+2\gamma}
\int_{|w|\le\e^{-1}}
|w|^{2\gamma}F_H(w)\dd w.
$$
Applying \eqref{eq:localmoment} with $R=\e^{-1}$ yields
$$
J_1\le
C
\begin{cases}
\e^{2H+2\gamma}+\e,
& 2\gamma+2H-1>0,\\
\e(1+|\log\e|),
& 2\gamma+2H-1=0,\\
\e^{2H+2\gamma},
& 2\gamma+2H-1<0.
\end{cases}
$$
 For $J_2$, since $|w|>\e^{-1}$,
$$
\e^{2\gamma}|w|^{2\gamma}\wedge1=1,
\qquad
(1+\e|w|)^a\le C\e^a(1+|w|)^a.
$$
Hence, by \eqref{eq:weightedtail},
$$
J_2
\le
C\e^{2H+a}
\int_{|w|>\e^{-1}}
(1+|w|)^aF_H(w)\dd w
\le
C\e.
$$

Combining the estimates for $J_1$ and $J_2$ proves
\eqref{eq:wlocalenergy}.
\end{proof}

\begin{lemma}\label{lem:singleenergy}
\begin{itemize}

\item[(i)]
Let $0\le a<1$ and $0<\e\le1$. Then
$$
\int_0^\infty\int_\R
|D_r(z,\e)|^2(1+|z|)^a
\dd z\dd r
\le
C\e.
$$

\item[(ii)]
Let $b\ge0$, $0<\gamma<1$, $0<\rho\le1$, and
$0<\e\le\sqrt{\rho}$. Then
$$
\int_\R
|D_\rho(z,\e)|(1+|z|)^b
(|z|^\gamma\wedge1)\dd z
\le
C\e\rho^{-(1-\gamma)/2}.
$$

\end{itemize}
\end{lemma}

\begin{proof}
For (i), the scaling $r=\e^2s$ and $z=\e w$ gives
$$
\begin{aligned}
 \int_0^\infty\int_\R
|D_r(z,\e)|^2(1+|z|)^a\dd z\dd r
=&\,
\e\int_0^\infty\int_\R
|D_s(w,1)|^2(1+\e|w|)^a\dd w\dd s\\
\le&\,
\e\int_0^\infty\int_\R
|D_s(w,1)|^2(1+|w|)^a\dd w\dd s.
\end{aligned}
$$
It therefore suffices to show that
\begin{equation}\label{eq:single-energy-target}
\int_0^\infty\int_\R
|D_s(w,1)|^2(1+|w|)^a\dd w\dd s<\infty.
\end{equation}
For $0<s\le1$,
$$
|D_s(w,1)|^2
\le
2\bigl(G_s(w+1)^2+G_s(w)^2\bigr),
$$
and Gaussian decay yields
$$
\int_\R
|D_s(w,1)|^2(1+|w|)^a\dd w
\le
Cs^{-1/2}.
$$
For $s\ge1$, using Jensen's inequality 
and  \eqref{eq:heat-kernel-derivative-scaling}, we have
$$
\begin{aligned}
\int_\R|D_s(w,1)|^2(1+|w|)^a\dd w
\le&\,
\int_0^1\int_\R
|G_s'(w+c)|^2(1+|w|)^a\dd w\dd c\\
\le&\,
Cs^{(a-3)/2}
\int_\R|G_1'(z)|^2(1+|z|)^a\dd z\\
\le&\,
Cs^{(a-3)/2}.
\end{aligned}
$$
Since $a<1$, both bounds are integrable over their respective ranges,
which proves \eqref{eq:single-energy-target}.

For (ii), using $|z|^\gamma\wedge1\le|z|^\gamma$ and the
fundamental theorem of calculus, we obtain
$$
\begin{aligned}
\int_\R
|D_\rho(z,\e)|(1+|z|)^b(|z|^\gamma\wedge1)\dd z
\le&\,
\e\int_0^1\int_\R
|G_\rho'(z+c\e)|(1+|z|)^b|z|^\gamma\dd z\dd c.
\end{aligned}
$$
Since $\e\le\sqrt{\rho}\le1$ and $0\le c\le1$,
$$
(1+|z|)^b\le C(1+|z+c\e|)^b,
\qquad
|z|^\gamma\le|z+c\e|^\gamma+\e^\gamma.
$$
Using the change of  variables $y=z+c\e$ and 
\eqref{eq:heat-kernel-derivative-scaling}, we obtain
$$
\begin{aligned}
 \int_\R
|D_\rho(z,\e)|(1+|z|)^b(|z|^\gamma\wedge1)\dd z 
\le&\,
C\e\int_\R
|G_\rho'(y)|(1+|y|)^b
\bigl(|y|^\gamma+\e^\gamma\bigr)\dd y\\
\le&\,
C\e\left(
\rho^{(\gamma-1)/2}
+\e^\gamma\rho^{-1/2}
\right)\\
\le&\,
C\e\rho^{-(1-\gamma)/2}.
\end{aligned}
$$ 

The proof is complete. 
\end{proof}

\end{document}